\documentclass[a4paper, 10pt]{amsart}
\usepackage{amsmath}
\usepackage{amssymb, color}
\usepackage{hyperref}

\theoremstyle{plain}
\newtheorem{theorem}{\bf Theorem}[section]
\newtheorem{proposition}[theorem]{\bf Proposition}
\newtheorem{lemma}[theorem]{\bf Lemma}

\theoremstyle{definition}

\newtheorem{remarks}[theorem]{\bf Remarks}

\newcommand{\N}{\mathbb N}
\newcommand{\Z}{\mathbb Z}

 \DeclareMathOperator{\ord}{ord}

 \DeclareMathOperator{\supp}{supp}

\renewcommand{\time}{\negthinspace \times \negthinspace}

\renewcommand{\t}{\, | \,}

\numberwithin{equation}{section}

\begin{document}

\title{The set of minimal distances in Krull monoids}

\address{University of Graz, NAWI Graz \\
Institute for Mathematics and Scientific Computing \\
Heinrichstra{\ss}e 36\\
8010 Graz, Austria}

\email{alfred.geroldinger@uni-graz.at, qinghai.zhong@uni-graz.at}

\author{Alfred Geroldinger  and Qinghai Zhong}

\thanks{This work was supported by
the Austrian Science Fund FWF, Project Number M1641-N26.}



\begin{abstract}
Let $H$ be a Krull monoid with finite class group $G$. Then every non-unit $a \in H$ can be written as a finite product of atoms, say $a=u_1 \cdot \ldots \cdot u_k$. The set $\mathsf L (a)$ of all possible factorization lengths $k$ is called the set of lengths of $a$. If $G$ is finite, then there is a constant $M \in \N$ such that all sets of lengths are almost arithmetical multiprogressions with bound $M$ and with difference $d \in \Delta^* (H)$, where $\Delta^* (H)$ denotes the set of minimal distances of $H$. We show that $\max \Delta^* (H) \le \max \{\exp (G)-2, \mathsf r (G)-1\}$  and that equality holds if every class of $G$ contains a prime divisor, which holds true for holomorphy rings in global fields.
\end{abstract}

\maketitle

%%%%%%%%%%%%%%%%%%%%%%%%%%%%%%%%%%%%%%%%%%%%%%%%%%%%%%%%%%%%%%%%%%%%%%%%%
%%                                      %%%%%%%%%%%%%%%
%%%%%%%%%%%%%%%%%%%%%%%%%%%%%%%%%%%%%%%%%%%%%%%%%%%%%%%%%%%%%%%%%%%%%%%%%
\bigskip
\section{Introduction}
\bigskip

Let $H$ be a Krull monoid with class group $G$ (we have in mind holomorphy rings in global fields and give more examples later). Then every non-unit of $H$ has a factorization as a finite product of atoms (or irreducible elements), and all these factorizations are unique (i.e., $H$ is factorial) if and only if $G$ is trivial. Otherwise, there are elements having factorizations which differ not only up to associates and up to the order of the factors. These phenomena are described by arithmetical invariants such as sets of lengths and sets of distances. We first recall some concepts and then we formulate a main result of the present paper.

For a finite nonempty set $L = \{m_1, \ldots, m_k\}$ of positive integers with $m_1 < \ldots < m_k$, we denote by $\Delta (L) = \{m_i-m_{i-1} \mid i \in [2,k] \}$ the set of distances of $L$. Thus $\Delta (L)=\emptyset$ if and only if $|L|\le 1$. If a non-unit $a \in H$ has a factorization $a = u_1 \cdot \ldots \cdot u_k$ into atoms $u_1, \ldots, u_k$, then $k$ is called the length of the factorization, and the set $\mathsf L_H (a) = \mathsf L (a)$ of all possible $k$ is called the set of lengths of $a$. If there is an element $a \in H$ with $|\mathsf L (a)|>1$, then it immediately follows  that  $|\mathsf L (a^n)| > n$ for every $n \in \N$. Since $H$ is  Krull, every non-unit has a factorization into atoms and all sets of lengths are finite. The set of distances $\Delta (H)$ is the union of all sets $\Delta (\mathsf L (a))$ over all non-units $a \in H$. Thus, by definition, $\Delta (H)=\emptyset$ if and only if $|\mathsf L (a)|=1$ for all non-units $a\in H$, and $\Delta (H)=\{d\}$ if and only if $\mathsf L (a)$ is an arithmetical progression with difference $d$ for all non-units $a \in H$.
The set of minimal distances $\Delta^* (H)$ is defined as
\[
\Delta^* (H) = \{ \min \Delta (S) \mid S \subset H \ \text{is a divisor-closed submonoid with} \ \Delta (S) \ne \emptyset \} \,.
\]
By definition, we have $\Delta^* (H) \subset \Delta (H)$, and $\Delta^* (H) = \emptyset$ if and only if $\Delta (H) = \emptyset$. If the class group $G$ is finite, then $\Delta (H)$ is finite and sets of lengths have a well-defined structure which is given in the next theorem (\cite[Chapter 4.7]{Ge-HK06a}).

\medskip
\noindent
{\bf Theorem A.} {\it Let $H$ be a Krull monoid with finite class group. Then there is  a constant $M \in \N$ such that the set of lengths $\mathsf L (a)$ of any non-unit $a\in H$ is an {\rm AAMP} $($almost arithmetical multiprogression$)$ with difference $d \in \Delta^* (H)$ and bound $M$.}

\medskip
The structural description given above is best possible (\cite{Sc09a}).
The set of minimal distances $\Delta^* (H)$ has been studied by Chapman, Geroldinger, Halter-Koch, Hamidoune, Plagne, Smith, Schmid, and others and there are a variety of results. We refer the reader to the monograph \cite[Chapter 6.8]{Ge-HK06a} for an overview and mention some results which have appeared since then. Suppose that $G$ is finite and that every class contains a prime divisor.
Then the set of distances $\Delta (H)$ is an interval (\cite{Ge-Yu12b}). A simple example shows that the interval $[1, \mathsf r (G)-1]$ is contained in $\Delta^* (H)$ (Lemma \ref{3.1}) and thus, by Theorem \ref{1.1} below,  $\Delta^* (H)$ is an interval too if $\mathsf r (G) \ge \exp (G)-1$. Cyclic groups are in sharp contrast to this. Indeed, if $G$ is cyclic with $|G|>3$, then $\max \big( \Delta^* (H) \setminus \{|G|-2\}\big) = \lfloor \frac{|G|}{2} \rfloor -1$ (\cite{Ge-Ha02}).
A detailed study of the structure of $\Delta^* (H)$ in case of cyclic groups is given in a recent paper by Plagne and Schmid \cite{Pl-Sc16a}.

The goal of the present paper is to study the maximum of $\Delta^* (H)$, and here is the main direct result.

\medskip
\begin{theorem} \label{1.1}
Let $H$ be a Krull monoid with class group $G$.
\begin{enumerate}
\item If \ $|G| \le 2$, then $\Delta^* (H) = \emptyset$.

\smallskip
\item If \ $2 < |G| < \infty$, then $\max \Delta^* (H) \le \max \{\exp (G)-2, \mathsf r (G) -1\}$ where $\mathsf r (G)$ denotes the rank of $G$.

\smallskip
\item Suppose that every class contains a prime divisor. If  $G$ is infinite, then $\Delta^* (H) = \mathbb N$. \newline If \ $2 < |G| < \infty$, then $\max \Delta^* (H) = \max \{\exp (G)-2, \mathsf r (G)-1\}$.
\end{enumerate}
\end{theorem}

\smallskip
Theorem \ref{1.1} will be complemented by an associated inverse result (Theorem \ref{4.5})  describing how $\max \Delta^* (H)$ is realized and disproving a former conjecture (Remark \ref{4.6}). Both the direct as well as the inverse result have number theoretic relevance beyond the occurrence in Theorem A. Indeed, they are key tools in the characterization of those Krull monoids whose systems of sets of lengths are  closed under set addition (\cite{Ge-Sc15b}),  in the study of arithmetical characterizations of class groups via sets of lengths (\cite[Chapter 7.3]{Ge-HK06a}, \cite{Sc09c, Ge-Sc16a}), as well as  in the asymptotic study of counting functions associated to  periods of sets of lengths (\cite{Sc08d} and \cite[Theorem 9.4.10]{Ge-HK06a}).

In Section \ref{2} we gather the required background from the theory of Krull monoids and from Additive Combinatorics. In particular, we outline that the set of minimal distances of $H$ equals the set of minimal distances of an associated monoid of zero-sum sequences (Lemma \ref{2.1}) and that  therefore it can be studied with methods from Additive Combinatorics. The proof of Theorem \ref{1.1} will be given in Section \ref{3} and the associated inverse result will be given in Section \ref{4}.

\medskip
\section{Background on Krull monoids and on Additive Combinatorics} \label{2}
\medskip

We denote by $\N$ the set of positive integers, and, for $a, b \in \Z$, we denote by $[a,b]=\{x \in \Z \mid a \le x \le b\}$ the discrete, finite interval between $a$ and $b$. We use the convention that $\max \emptyset = 0$. By a {\it monoid}, we mean a commutative semigroup with identity that satisfies the cancellation laws. If $H$ is a monoid, then $H^{\times}$ denotes the unit group, $\mathsf q (H)$ the quotient group, and $\mathcal A (H)$  the set of atoms (or irreducible elements) of $H$. A submonoid $S \subset H$ is called {\it divisor-closed} if $a \in S$, $b \in H$, and $b$ divides $a$ imply that $b \in S$.
A monoid $H$ is said to be
\begin{itemize}
\item {\it atomic} if every non-unit can be written as a finite product of atoms.

\item {\it factorial} if it is atomic and every atom is prime.

\item {\it half-factorial} if it is atomic and $|\mathsf L (a)|=1$ for each non-unit $a \in H$ (equivalently, $\Delta (H)=\emptyset$).

\item {\it decomposable} if there exist submonoids $H_1, H_2$ with $H_i \not\subset H^{\times}$ for $i \in [1,2]$ such that $H = H_1 \times H_2$ (and $H$ is called {\it indecomposable} else).
\end{itemize}
A monoid $F$ is factorial with $F^{\times} = \{1\}$ if and only if it is free abelian. If this holds, then the set of primes $P \subset F$ is a basis of $F$, we write $F = \mathcal F (P)$, and every $a \in F$ has a representation of the form
\[
a = \prod_{p \in P} p^{\mathsf v_p (a)} \quad \text{with} \ \mathsf v_p (a) \in \N_0 \quad \text{and} \quad \mathsf v_p (a) = 0 \ \text{for almost all} \ p \in P \,.
\]

A monoid homomorphism \ $\theta \colon H \to B$ is called a \ {\it
transfer homomorphism} \ if it has the following properties:
\begin{enumerate}
\item[]
\begin{enumerate}
\item[{\bf (T\,1)\,}] $B = \theta(H) B^\times$ \ and \ $\theta
^{-1} (B^\times) = H^\times$.

\smallskip

\item[{\bf (T\,2)\,}] If $u \in H$, \ $b,\,c \in B$ \ and \ $\theta
(u) = bc$, then there exist \ $v,\,w \in H$ \ such that \ $u = vw$, \
$\theta (v) \simeq b$ \ and \ $\theta (w) \simeq c$.
\end{enumerate}
\end{enumerate}
If $H$ and $B$ are atomic monoids and $\theta \colon H \to B$ is a transfer homomorphism, then  (see \cite[Chapter 3.2]{Ge-HK06a})
\[
\mathsf L_H (a) = \mathsf L_B ( \theta (a) ) \ \text{ for all $a \in H$}, \quad \Delta (H) = \Delta (B), \quad \text{and} \quad
\Delta^* (H) = \Delta^* (B) \,.
\]

\medskip
\noindent
{\bf Krull monoids.} A monoid $H$ is said to be a {\it Krull monoid} if it satisfies the following two conditions:
\begin{enumerate}
\item[(a)] There exists a monoid homomorphism $\varphi \colon H \to F = \mathcal F (P)$ into a free abelian monoid $F$ such that $a \t b$ in $H$ if and only if $\varphi (a) \t \varphi (b)$ in $F$.

\item[(b)] For every $p \in P$, there exists a finite subset $E \subset H$ such that $p = \gcd \big( \varphi (E) \big)$.
\end{enumerate}
Let $H$ be a Krull monoid and $\varphi \colon H \to \mathcal F (P)$ a homomorphism satisfying Properties (a) and (b). Then $\varphi$ is called a divisor theory of $H$, $G= \mathsf q (F)/ \mathsf q (\varphi (H))$ is the class group, and $G_P = \{ [p] = p \mathsf q (\varphi (H))) \mid p \in P\} \subset G$ the set of classes containing prime divisors. The class group will be written additively, and the tuple $(G, G_P)$ are uniquely determined by $H$. To provide some examples of Krull monoids, we recall that a domain is a Krull domain if and only if its multiplicative monoid of nonzero elements is a Krull monoid, and that a noetherian domain is Krull if and only if it is integrally closed. Rings of integers, holomorphy rings in algebraic function fields, and regular congruence monoids in these domains are Krull monoids with finite class group such that every class contains a prime divisor (\cite{Ge-HK04a}, \cite[Chapter 2.11]{Ge-HK06a}). For monoids of modules and monoid domains which are Krull we refer to \cite{Ki-Pa01, Ch11a, Ba-Wi13a, Ba-Ge14b}.

Next we introduce Krull monoids having a combinatorial flavor which are used to model arbitrary Krull monoids. Let $G$ be an additively written abelian group and $G_0 \subset G$ a subset. An element $S = g_1 \cdot \ldots \cdot g_l \in \mathcal F (G_0)$ is called a {\it sequence} over $G_0$, $\sigma (S) = g_1+ \ldots+g_l$ is called its sum,  $|S|=l$ its length, and
$\mathsf h(S)=\max \{\mathsf v_g(S)\mid g\in \supp(S)\}$ the maximal multiplicity of $S$.
The monoid
\[
\mathcal B (G_0) = \{S \in \mathcal F (G_0) \mid \sigma (S)=0 \}
\]
is a Krull monoid, called the {\it monoid of zero-sum sequences} over $G_0$.  Its significance for the study of general Krull monoids is summarized in the following lemma (see \cite[Theorem 3.4.10 and Proposition 4.3.13]{Ge-HK06a}).

\medskip
\begin{lemma} \label{2.1}
Let $H$ be a Krull monoid, $\varphi \colon H \to D = \mathcal F (P)$ a divisor theory with class group $G$ and $G_P \subset G$ the set of classes containing prime divisors. Let $\widetilde{ \boldsymbol \beta} \colon D \to \mathcal F (G_P)$ denote the unique homomorphism defined by $\widetilde{ \boldsymbol \beta} (p) = [p]$ for all $p \in P$. Then the homomorphism $\boldsymbol \beta = \widetilde{ \boldsymbol \beta} \circ \varphi \colon H \to \mathcal B (G_P)$ is a transfer homomorphism. In particular, we have
\[
\Delta^* (H) = \Delta^* \big( \mathcal B (G_P) \big) = \big\{ \min \Delta \big( \mathcal B (G_0) \big) \mid G_0 \subset G_P \ \text{is a subset such that} \ \mathcal B (G_0) \ \text{is not half-factorial} \big\} \,.
\]
\end{lemma}
Thus $\Delta^* (H)$ can be studied in an associated monoid of zero-sum sequences and can thus be tackled by methods from Additive Combinatorics. Such transfer results to monoids of zero-sum sequences are not restricted to Krull monoids, but they do exist also from certain seminormal weakly Krull monoids and from certain maximal orders in central simple algebras over global fields. We do not outline this here but refer to \cite[Theorem 1.1]{Sm13a}, \cite{Ge-Ka-Re15a}, and \cite[Section 7]{Ba-Sm15}.

\medskip
\noindent
{\bf Zero-Sum Theory} is a vivid subfield of Additive Combinatorics (see the  monograph \cite{Gr13a},  the survey \cite{Ge09a}, and for a sample of recent papers on direct and inverse zero-sum problems with a strong number theoretical flavor see \cite{Gi08b,G-L-P-P-W11a, Ha11a, Ze-Yu11a, Ga-Pe-Zh13a}). We gather together the concepts needed in the sequel.

Let $G$ be a finite abelian group and $G_0 \subset G$ a subset. Then $\langle G_0 \rangle \subset G$ denotes
the subgroup generated by $G_0$.    A family $(e_i)_{i \in I}$ of elements of $G$ is said to be \
{\it independent} \ if $e_i \ne 0$ for all $i \in I$ and, for every
family $(m_i)_{i \in I} \in \Z^{(I)}$,
\[
\sum_{i \in I} m_ie_i =0 \qquad \text{implies} \qquad m_i e_i =0 \quad \text{for all} \quad i \in I\,.
\]
The family $(e_i)_{i \in I}$ is called a {\it basis} for $G$ if
 $G = \bigoplus_{i \in I} \langle e_i \rangle$. The set $G_0$ is said to be independent if the tuple $(g)_{g \in G_0}$ is independent. If for a prime $p \in \mathbb P$, $\mathsf
r_p (G)$ is the $p$-rank of $G$, then
\[
\mathsf r (G) = \max \{ \mathsf r_p (G) \mid p \in \mathbb P\} \ \ \text{is the {\it rank} of $G$  and} \ \
\mathsf r^* (G) = \sum_{p \in \mathbb P} \mathsf r_p (G) \ \text{ is the {\it total rank} of $G$} \,.
\]
The monoid $\mathcal B (G_0)$ of zero-sum sequences over $G_0$ is a finitely generated Krull monoid. It is traditional to set
\[
\mathcal A (G_0) := \mathcal A \big(\mathcal B (G_0)\big), \ \Delta (G_0):= \Delta \big( \mathcal B (G_0)\big) , \ \text{and} \   \Delta^* (G_0) :=  \Delta^* \big( \mathcal B (G_0)\big) \,.
\]
Clearly, the atoms of $\mathcal B (G_0)$ are precisely the minimal zero-sum sequences over $G_0$. The set $\mathcal A (G_0)$ is finite, and $\mathsf D (G_0) = \max \{ |S| \mid S \in \mathcal A (G_0)\}$ is the {\it Davenport constant} of $G_0$.
The set $G_0$ is called
\begin{itemize}
\item {\it half-factorial} \ if  the monoid $\mathcal B (G_0)$ is half-factorial (equivalently, $\Delta (G_0) = \emptyset$).

\item {\it non-half-factorial} \ if the monoid $\mathcal B (G_0)$ is not half-factorial (equivalently, $\Delta (G_0) \ne \emptyset$).

\item {\it minimal non-half-factorial} \ if $\Delta (G_0) \ne \emptyset$ but every proper subset is half-factorial.

\item {\it $($in$)$decomposable} \ if the monoid $\mathcal B (G_0)$ is (in)decomposable.
\end{itemize}
(Maximal) half-factorial and (minimal) non-half-factorial subsets have found a lot of attention in the literature (see \cite{Ge-Go03, Sc05c, Pl-Sc05a, Pl-Sc05b, Sc06a, Ch-Ch-Sm07b,Ch-Go-Pe14a}), and cross numbers are a crucial tool for their study.
For a sequence $S = g_1 \cdot \ldots \cdot g_l \in \mathcal F (G_0)$, we call
\[
\begin{aligned}
\mathsf k (S) & = \sum_{i=1}^l \frac{1}{\ord (g_i)} \ \in \mathbb Q_{\ge 0} \quad \text{the {\it cross number} of $S$, and } \\
\mathsf K (G_0) & = \max \{ \mathsf k (S) \mid S \in \mathcal A (G_0) \} \quad \text{the {\it cross number} of $G_0$}.
\end{aligned}
\]
The following simple result (\cite[Proposition 6.7.3]{Ge-HK06a}) will be used throughout the paper without further mention.

\medskip
\begin{lemma} \label{2.3}
Let $G$ be a finite abelian group and $G_0 \subset G$ a subset. Then
the following statements are equivalent{\rm \,:}
\begin{enumerate}
\item[(a)]
$G_0$ is half-factorial.

\smallskip

\item[(b)]
$\mathsf k (U) = 1$ for every $U \in \mathcal A (G_0)$.

\smallskip

\item[(c)]
$\mathsf L (B) = \{ \mathsf k (B) \}$ for every $B \in
      \mathcal B (G_0)$.
\end{enumerate}
\end{lemma}

\medskip
\section{Direct results on $\Delta^* (H)$} \label{3}
\medskip

We start with a basic well-known lemma (see \cite[Chapter 6.8]{Ge-HK06a}).

\medskip
\begin{lemma} \label{3.1}
Let $G$ be a finite abelian group with $|G| > 2$.
\begin{enumerate}
\item If $g \in G$ with $\ord (g) > 2$, then $\ord (g) - 2 \in
\Delta^* (G)$. In particular, $\exp (G)-2 \in \Delta^* (G)$.

\smallskip
\item If $\mathsf r (G) \ge 2$, then $[1, \mathsf r (G)-1] \subset
\Delta^* (G)$.

\smallskip
\item Let $G_0 \subset G$ a subset.
      \begin{enumerate}
      \item If there exists a $U \in \mathcal A (G_0)$ with $\mathsf k (U) < 1$, then $\min \Delta (G_0) \le \exp (G)-2$.

      \smallskip
      \item If $\mathsf k (U) \ge 1$ for all $U \in \mathcal A (G_0)$, then $\min \Delta (G_0) \le |G_0|-2$.
      \end{enumerate}
\end{enumerate}
\end{lemma}

\begin{proof}
1. Let $g \in G$ with $\ord (g) = n > 2$ and set $G_0 = \{g, -g
\}$. Then $\mathcal A (G_0) = \{g^n, (-g)^n, \big( (-g)g \big) \}$,
$\Delta (G_0) = \{n-2\}$, and hence $\min \Delta (G_0) = n-2$.

\smallskip
2. Let $s \in [2, \mathsf r (G)]$. Then there is a prime $p \in
\mathbb P$ such that $C_p^s$ is isomorphic to a subgroup of $G$, and
it suffices to show that $s-1 \in \Delta^* (C_p^s)$. Let $(e_1,
\ldots, e_s)$ be a basis of $C_p^s$ and set $e_0 = e_1 + \ldots +
e_s$ and $G_0 = \{e_0, \ldots, e_s \}$. Then a simple calculation
(details can be found in \cite[Proposition 6.8.1]{Ge-HK06a}) shows
that $\Delta (G_0) = \{s-1\}$ and hence $\min \Delta (G_0) = s-1$.

\smallskip
3.(a) Let $U = g_1 \cdot \ldots \cdot g_l \in \mathcal A (G_0)$ with
$\mathsf k (U) < 1$ and $n = \exp (G)$ (note that $\mathsf k (U) <
1$ implies $U \ne 0$, $l \ge 2$ and $\mathsf k (U) > \frac{1}{n}$).
Then $U_i = g_i^{\ord (g_i)} \in \mathcal A (G_0)$ for all $i \in
[1, l]$, and
\[
U^n = \prod_{i=1}^l U_i^{n/ \ord (g_i)}
\]
implies that $n \mathsf k (U) = \sum_{i=1}^l \frac{n}{\ord (g_i)}
\in \mathsf L (U^n)$. Since $\mathsf k (U) < 1$, we have $n \mathsf
k (U) \in [2, n-1]$ and $\min \Delta (G_0) \le n - n \mathsf k (U)
\in [1, n-2]$.

3.(b) The proof is similar to that of 3.(a), see \cite[Lemma 6.8.6]{Ge-HK06a} for details.
\end{proof}

\medskip
Lemma \ref{3.1}.3 motivates the following  definitions (see \cite{Sc08d, Sc09c}). A subset $G_0 \subset G$ is called an LCN-set ({\it large cross number set}) if $\mathsf k (U) \ge 1$ for each $U \in \mathcal A (G_0)$ and
\[
\mathsf m (G) = \max \big\{ \min \Delta (G_0) \mid G_0 \subset G \
 \text{is a non-half-factorial LCN-set}
 \big\} \,.
\]
Clearly, if $G$ has a non-half-factorial LCN-set, then $|G|\ge4$.
The following result (due to  Schmid \cite{Sc09c}) is crucial for our approach.

\medskip
\begin{proposition} \label{3.2}
Let $G$ be a finite abelian group with $|G| > 2$. Then
\[
\max \Delta^* (G) = \max \{ \exp (G)-2, \mathsf m (G) \} \
\text{and} \ \mathsf m (G) \le \max \{ \mathsf r^* (G)-1, \mathsf K
(G)-1 \} \,.
\]
If $G$ is a $p$-group, then $\mathsf m (G) = \mathsf r (G)-1$ and
thus $\max \Delta^* (G) = \max \{ \exp (G)-2, \mathsf r (G)-1\}$.
\end{proposition}

\begin{proof}
See \cite[Theorem 3.1, Lemma 3.3.(4), and Proposition 3.6]{Sc09c}.
\end{proof}

\medskip
\begin{lemma} \label{3.3}
Let $G$ be a finite abelian group and $G_0 \subset G$  a subset.
\begin{enumerate}
\item The following statements are equivalent{\rm \;:}
      \begin{enumerate}
      \item $G_0$ is decomposable.

      \item There are nonempty subsets $G_1, G_2 \subset G_0$  such that $G_0=G_1 \uplus G_2$ and $\mathcal B(G_0)=\mathcal B(G_1)\time \mathcal B(G_2)$.

      \item There are nonempty subsets $G_1, G_2 \subset G_0$ such that $G_0=G_1 \uplus G_2$ and $\mathcal A(G_0) = \mathcal A(G_1) \uplus \mathcal A(G_2)$.

      \item There are nonempty subsets $G_1, G_2 \subset G_0$ such that $\langle G_0 \rangle = \langle G_1 \rangle \oplus \langle G_2 \rangle$.
      \end{enumerate}

\smallskip
\item If $G_0$ is minimal non-half-factorial, then $G_0$ is indecomposable.
\end{enumerate}
\end{lemma}

\begin{proof}
1. See \cite[Lemma 3.7]{Sc04a} and \cite[Lemma 3.2]{Ba-Ge14b}.

2. This follows immediately from 1.(b).
\end{proof}

\medskip
\begin{lemma} \label{3.4}
Let $G$ be a finite abelian group and $G_0 \subset G$  a subset.
\begin{enumerate}
\item For each $g \in G_0$,
      \[
      \begin{aligned}
      &\gcd \big( \{ \mathsf v_g (B) \mid B \in \mathcal B (G_0) \} \big)  =  \gcd \big( \{ \mathsf v_g (A) \mid A \in \mathcal A (G_0) \} \big)\\ =&
      \min  \big( \{ \mathsf v_g (A) \mid \mathsf v_g (A)>0,  A  \in \mathcal A (G_0) \} \big)   =  \min \big( \{ \mathsf v_g (B) \mid \mathsf v_g (B) > 0, B \in \mathcal B (G_0) \} \big) \\ =& \min  \big( \{ k \in \N  \mid kg \in \langle G_0 \setminus \{g\} \rangle  \} \big)=\gcd \big(  \{ k \in \N  \mid kg \in \langle G_0 \setminus \{g\} \rangle  \} \big) \,.
      \end{aligned}
      \]
In particular, $\min  \big( \{ k \in \N  \mid kg \in \langle G_0 \setminus \{g\} \rangle  \} \big)$ divides $\ord (g)$.

\smallskip
\item Suppose that for any $h\in G_0$, we have that  $h\not\in \langle G_0\setminus \{h,\ h'\}\rangle $ for any $h'\in G_0\setminus \{h\}$.
 Then for any atom $A$ with $\supp(A)\subsetneq G_0$ and any $h \in \supp (A)$, we have $\gcd (\mathsf v_h(A), \ord(h) )>1$.

\smallskip
\item If $G_0$ is minimal non-half-factorial, then there exists a minimal non-half-factorial subset $G_0^* \subset G$ with $|G_0|=|G_0^*|$ and a transfer homomorphism $\theta \colon \mathcal B (G_0) \to \mathcal B (G_0^*)$ such that  the following  properties are satisfied{\rm \,:}
    \begin{enumerate}
    \item   For each $g \in G_0^*$, we have $g \in \langle G_0^* \setminus \{g\} \rangle$.

    \smallskip
    \item  For each $B \in \mathcal B (G_0)$, we have $\mathsf k (B) = \mathsf k \big( \theta (B) \big)$.
  \smallskip
  \item If $G_0^*$ has the property that  for each $h\in G_0^*$, $h\not\in \langle E\rangle$ for any $E\subsetneq G_0^*\setminus \{h\}$, then $G_0$ also has the property.

   \smallskip
    \item If $G_0^*$ has the property that  there exists $h\in G_0^*$, such that  $G_0^*\setminus\{h\}$ is independent, then $G_0$ also has the property.
    \end{enumerate}
\end{enumerate}
\end{lemma}

\begin{proof}
1. Let $g \in G_0$ and let $\gamma_1, \ldots, \gamma_6$ denote the six terms in the given order of  the asserted equation. By definition, it follows that $\gamma_1 \le \gamma_2 \le \gamma_3$. Since $\{ \mathsf v_g (B) \mid B \in \mathcal B (G_0) \}= \{ k \in \N  \mid kg \in \langle G_0 \setminus \{g\} \rangle  \}$, we have that $\gamma_1=\gamma_6$ and $\gamma_4=\gamma_5$. Therefore we only need to show $\gamma_3\le \gamma_4$ and $\gamma_4\le \gamma_1$.

To show that $\gamma_3 \le \gamma_4$, let $B \in \mathcal B (G_0)$ such that $\mathsf v_g (B) = \gamma_4$. Suppose that $B = A_1 \cdot \ldots \cdot A_s$ with $s \in \N$ and $A_1, \ldots, A_s \in \mathcal A (G_0)$. Then $\mathsf v_g (B) = \mathsf v_g (A_1) + \ldots + \mathsf v_g (A_s)$. The minimality of $\mathsf v_g (B)$ implies that there is precisely one $i \in [1,s]$ with $\mathsf v_g (A_i) = \mathsf v_g (B)$ and $\mathsf v_g (A_j) = 0$ for all $j \in [1, s] \setminus \{i\}$. Thus $\gamma_3 \le \mathsf v_g (A_i) = \mathsf v_g (B) = \gamma_4$.

Next we show that  $\gamma_4 \le \gamma_1$. There are $s \in \N$, $r \in [1, s]$, $U_1, \ldots, U_s \in \mathcal B (G_0)$, and $k_1, \ldots, k_s \in \N$ such that
\[
\begin{aligned}
\gamma_1 & = k_1 \mathsf v_g (U_s) + \ldots + k_r \mathsf v_g (U_r) - k_{r+1} \mathsf v_g (U_{r+1}) - \ldots - k_s \mathsf v_g (U_s) \\
 & = \mathsf v_g (U_1^{k_1} \cdot \ldots \cdot U_r^{k_r}) - \mathsf v_g (U_{r+1}^{k_{r+1}} \cdot \ldots \cdot U_s^{k_s}) \,.
\end{aligned}
\]
Setting $B_1 = U_1^{k_1} \cdot \ldots \cdot U_r^{k_r}$, $B_2 = U_{r+1}^{k_{r+1}} \cdot \ldots \cdot U_s^{k_s}$, and $B_3 = \prod_{h \in G_0 \setminus \{g\}} h^{|B_2|}$ we obtain that $B_1 B_2^{-1} B_3 \in \mathcal B (G_0)$ and
\[
\gamma_1 = \mathsf v_g (B_1) - \mathsf v_g (B_2) = \mathsf v_g ( B_1 B_2^{-1} B_3 ) \ge \gamma_4 \,.
\]
In particular, $\gamma_5= \gamma_2$ divides $\ord (g)$ because $g^{\ord (g)} \in \mathcal A (G_0)$.

\smallskip
2. Assume to the contrary that there are $A$ and $h$ as above such that $\gcd (\mathsf v_h(A), \ord(h) ) = 1$. Choose $h' \in G_0 \setminus \supp (A)$, then $h \in \langle \supp (A) \setminus \{h\} \rangle \subset \langle G_0 \setminus \{h, h'\} \rangle$, a contradiction.

\smallskip
3. By \cite[Theorem 6.7.11]{Ge-HK06a}, there are a subset $G_0^* \subset G$ satisfying Property (a) and a transfer homomorphism $\theta \colon \mathcal B (G_0) \to \mathcal B (G_0^*)$. Moreover, the transfer homomorphism $\theta$ is a composition of transfer homomorphisms $\theta'$ of the following form:
\begin{itemize}
\item Let $g \in G_0$, $m=\min \bigl\{ k \in \N \mid kg \in \langle G_0\setminus\{g\} \rangle \bigr\}$, $G_0' = G_0 \setminus \{g\} \cup \{mg\}$, and
\[
\theta' \colon \mathcal B (G_0) \to \mathcal B ( G_0' ) \,, \quad \text{defined by} \quad \theta' (B)
= g^{-\mathsf v_g(B)}(mg)^{\mathsf v_g(B)/m}B\,,
\]
It is outlined that $m \t \mathsf v_g (B)$ and that $m \t \ord (g)$.
\end{itemize}
Therefore it is sufficient to show that $|G_0| = |G_0'|$ and that $\theta'$ satisfies Properties (b) - (d).

\smallskip
(i) By definition, we have $\mathsf k (B) = \mathsf k ( \theta' (B))$ for all $B \in \mathcal B (G_0)$.

\smallskip
(ii)  Since $G_0$ is a minimal non-half-factorial set, the same is true for $G_0'$ by \cite[Lemma 6.8.9]{Ge-HK06a}.
  If  $mg\in G_0\setminus \{g\}$, then $G_0'\subsetneq G_0$ would be non-half-factorial,  a contradiction to the minimality of $G_0$.  It follows that  $mg\not\in G_0\setminus \{g\}$, which implies that $|G_0'|=|G_0|$.

\smallskip
(iii) We set $G_0 = \{g=g_1,  \ldots, g_k\}$ (note that $k \ge 2$), $G_0' = \{mg, g_2, \ldots, g_k\}$, and suppose that     $h\not\in \langle E\rangle$ for each $h\in G_0'$ and for any $E\subsetneq G_0'\setminus \{h\}$. Assume to the contrary that there exist $h\in G_0$ and $E\subsetneq G_0\setminus \{h\}$ such that $h\in\langle E\rangle$.  If $h=g$, then $mg\in\langle E\rangle $, a contradiction.

Suppose that $h \ne g$, say  $h = g_k \in \langle E \rangle$ with $E \subsetneq \{g, g_2, \ldots, g_{k-1} \}$.  If $g\not\in E$, then $E\subsetneq G_0'\setminus\{mg\}$, a contradiction. Thus  $g\in E$, and we set $E'=E\setminus\{g\}\cup \{mg\}$. Since $h\in\langle E\rangle$, we have that $h=\sum_{x\in E\setminus\{g\}}t_xx+tg$ where $t_x,t\in \Z$. Thus $tg=h-\sum_{x\in E\setminus\{g\}}t_xx\in \langle E\cup \{h\}\setminus\{g\}\rangle\subset \langle G_0\setminus \{g\}\rangle.$  By 1., we obtain that $m\t t$ and hence $h=\sum_{x\in E\setminus\{g\}}t_xx+\frac{t}{m}mg\in \langle E'\rangle$, a contradiction.

(iv)  We set $G_0 = \{g=g_1,  \ldots, g_k\}$, $G_0' = \{mg, g_2, \ldots, g_k\}$, and suppose that    there exists $h\in G_0'$ such that  $G_0'\setminus\{h\}$ is independent. If $h=mg$, then $G_0\setminus\{g\}=G_0'\setminus\{h\}$ is independent. Suppose that  $h\not=mg$, say  $h = g_k$. Then $\{mg, g_2, \ldots, g_{k-1}\}$ is independent and  assume to the contrary that  $G_0\setminus\{h\} = \{g, g_2,  \ldots, g_{k-1}\}$ is not independent. Then there exist  $t_1,\ldots,t_{k-1}\in \Z$  such that $t_1g+t_2g_2+\ldots+t_{k-1}g_{k-1}=0$ but $t_ig_i \ne 0$ for at least one $i \in [1,k-1]$. This implies that $t_1g\in \langle g_2,\ldots,g_{k-1}\rangle \subset \langle G_0\setminus\{g\}\rangle$.  By 1., we obtain that $m\t t_1$ and hence $\frac{t_1}{m}mg+t_2g_2+\ldots+t_{k-1}g_{k-1}=0$, a contradiction to
$\{mg, g_2, \ldots, g_{k-1}\}$ is independent.

\end{proof}

\medskip
\begin{lemma} \label{3.5}
Let $G$ be a finite abelian  group and $G_0 \subset G$  a subset with $|G_0|\ge \mathsf r (G)+2$ such that the following two properties are satisfied{\rm \,:}
\begin{enumerate}
\item[(a)] For any $h\in G_0$, $G_0\setminus \{h\}$ is half-factorial and $h\not\in \langle G_0\setminus \{h,\ h'\}\rangle $ for any $h'\in G_0\setminus \{h\}$.
\item[(b)] There exists an element $g\in G_0$ such that $g\in \langle G_0\setminus \{g\} \rangle $ and $\ord(g)$ is not a  prime power.
\end{enumerate}
Then $|G_0|\le \exp(G)-2$.
\end{lemma}

\begin{proof}
We set $\exp (G)=n = p_1^{k_1} \cdot \ldots \cdot p_t^{k_t}$, where $t\ge 2, k_1, \ldots, k_t \in \N$ and $p_1, \ldots, p_t$ are distinct primes. By Lemma \ref{3.4}.2, we know that for any atom $A$ with $\supp(A)\subsetneq G_0$ and any $h \in \supp (A)$, we have $\gcd (\mathsf v_h(A), \ord(h) )>1$. In particular,
 \begin{equation}\label{e1}
 \mathsf v_h(A)\ge 2 \qquad\text{ for each $h\in \supp(A)$. }
 \end{equation}

We continue with the following assertion.

\begin{enumerate}
\item[{\bf A.}\,] For each $\nu \in [1, t]$ with $p_{\nu} \t \ord (g)$, there is an atom $U_{\nu} \in \mathcal A (G_0)$ such that $\mathsf v_g(U_{\nu}) \t \frac{n}{p_{\nu}^{k_{\nu}}}$, $\mathsf k(U_{\nu})=1$, and  $|\supp(U_{\nu})\setminus\{g\}|\le \frac{n-\mathsf v_g(U_{\nu})}{2}$.
\end{enumerate}

{\it Proof of \,{\bf A}}.\ Let $\nu \in [1,t]$ with $p_{\nu} \t \ord (g)$.
Since $g\in \langle G_0\setminus \{g\} \rangle$ and $t \ge 2$, it follows that $0 \ne \frac{n}{p_{\nu}^{k_{\nu}}}g\in G_{\nu} = \langle \frac{n}{p_{\nu}^{k_{\nu}}}h \mid h\in G_0\setminus \{g\}\rangle$. Obviously, $G_{\nu}$ is a $p_{\nu}$-group.
Let $E_{\nu} \subset G_0\setminus \{g\}$ be minimal such that $\frac{n}{p_{\nu}^{k_{\nu}}}g \in  \langle \frac{n}{p_{\nu}^{k_{\nu}}} E_{\nu}\rangle$. The minimality of $E_{\nu}$ implies that $|E_{\nu}| = |\frac{n}{p_{\nu}^{k_{\nu}}} E_{\nu}|$ and it implies that $\frac{n}{p_{\nu}^{k_\nu}} E_{\nu}$ is a minimal generating set of $G_{\nu}' := \langle \frac{n}{p_{\nu}^{k_{\nu}}} E_{\nu} \rangle$. Thus \cite[Lemma A.6.2]{Ge-HK06a} implies that $|\frac{n}{p_{\nu}^{k_\nu}} E_{\nu}| \le \mathsf r^* (G_{\nu}')$.  Putting all together we obtain that
\[
|E_{\nu}| = |\frac{n}{p_{\nu}^{k_{\nu}}} E_{\nu}| \le \mathsf r^* (G_{\nu}') = \mathsf r (G_{\nu}')  \le \mathsf r (G) \,.
\]
Let $d_{\nu} \in \N$ be  minimal  such that $d_{\nu}g\in \langle E_{\nu}\rangle$. By Lemma \ref{3.4}.1, $d_{\nu}\t \frac{n}{p_{\nu}^{k_{\nu}}}$ and there exists an atom $U_{\nu}$ such that $\mathsf v_h(U_{\nu})=d_{\nu}$ and $|\supp(U_{\nu})|\le |E_{\nu}|+1\le \mathsf r (G)+1 \le |G_0|-1$. Thus Property (a) implies that $\mathsf k(U_{\nu})=1$.
Let
\[
U_{\nu}=g^{\mathsf v_g(U_{\nu})}\prod_{h\in \supp(U_{\nu})\setminus\{g\}}h^{\mathsf v_{h}(U_{\nu})} \,.
\]
Since $\mathsf v_h(U_{\nu})\ge 2$ for each $h\in \supp(U_{\nu})\setminus\{g\}$ by  Equation \eqref{e1}, it follows that
\[
1=\mathsf k(U_{\nu})\ge \frac{\mathsf v_g(U_{\nu})}{n}+|\supp(U_{\nu})\setminus\{g\}|\frac{2}{n} \,,
\]
whence $|\supp(U_{\nu})\setminus\{g\}|\le \frac{n-\mathsf v_g(U_{\nu})}{2}$.
\qed{(Proof of {\bf A})}

\smallskip
Let $s \in \N$ be  minimal  such that there exists a nonempty subset $E\subsetneq G_0\setminus \{g\}$ with $sg\in \langle E \rangle$ and  let  $E \subsetneq G_0\setminus \{g\}$ be  minimal such that $sg\in \langle E \rangle$. By Lemma \ref{3.4}.1, there is an atom   $V$  with $\mathsf v_g(V)=s$ and $\supp(V)=\{g\}\cup E\subsetneq G_0$. Then
\[
1= \mathsf k (V) = \frac{s}{\ord (g)} + \sum_{h \in E} \frac{\mathsf v_h (V)}{\ord (h)} \,.
\]
By Equation \eqref{3.1}, we have that $\mathsf v_h(V)\ge 2$ for each $h\in E$  and hence the equation above implies that  $|E|\le \frac{n-s}{2}$.

\smallskip
\noindent CASE 1: \, $s$ is a power of a prime, say a power of $p_1$.
\smallskip

Let $E_1=\supp(U_1)\setminus\{g\}$.
Since  $\mathsf v_g(U_1) \t \frac{n}{p_1^{k_1}}$,
we have that $g\in \langle sg, \mathsf v_g(U_1)g \rangle \subset \langle E \cup E_1 \rangle$.
Property (a) implies  that $E\cup E_1=G_0\setminus \{g\}$, and thus
\[
|G_0| \le 1+|E|+|E_1|\le 1+ \frac{n-s}{2} + \frac{n-\mathsf v_g (U_1)}{2} = 1+ n-\frac{\mathsf v_g(U_1)+s}{2} \,.
\]
Since $\gcd(\mathsf v_g(U_1),s)=1$, it follows that $\mathsf v_g(U_1)+s\ge 5$, hence $|G_0|\le n-3/2$, and thus $|G_0|\le n-2$.

\smallskip
\noindent
CASE 2: \, $s$ is not a  prime power, say $p_1p_2 \t s$.
\smallskip

Then $s\ge 6$.
Let $d=\gcd(s,\mathsf v_g(U_1))$ and $E_1=\supp(U_1)\setminus\{g\}$, then $d<s$ and $dg \in \langle sg, \mathsf v_g(U_1)g \rangle \subset \langle E\cup E_1 \rangle \subset \langle G_0\setminus \{g\} \rangle$.
The minimality of $s$ implies  that $E\cup E_1=G_0\setminus \{g\}$, and thus
\[
|G_0| \le 1+|E|+|E_1|\le 1+ \frac{n-s}{2} + \frac{n-\mathsf v_g (U_1)}{2} = 1+ n-\frac{\mathsf v_g(U_1)+s}{2} \le n-3 \,. \qedhere
\]
\end{proof}

\medskip
\begin{lemma} \label{3.6}
Let $G$ be a finite abelian group  with $\exp(G)=n$.
Let $G_0 \subset G$ be a minimal non-half-factorial {\rm LCN}-set  and suppose that there is a
subset $G_2 \subset G_0$ such that $\langle G_2 \rangle = \langle G_0 \rangle$ and $|G_2| \le |G_0|-2$. Then $\min \Delta (G_0) \le  \max \{1, n-4\}$.
\end{lemma}

\begin{proof}
Assume to the contrary that $\min \Delta (G_0) \ge  \max \{2, n-3\}$.
By \cite[Corollary 3.1]{Sc05d}, the existence of the subset $G_2$ implies that $\mathsf k (U) \in \N$ for each $U \in \mathcal A (G_0)$ and
\[
\min \Delta (G_0) \t \gcd \big( \{ \mathsf k (A) - 1 \mid A \in \mathcal A (G_0) \} \big) \,.
\]
We set
\[
W_1=\{A\in \mathcal A (G_0) \mid \mathsf k (A)=1\} \quad \text{and} \quad
W_2=\{A\in \mathcal A (G_0) \mid \mathsf k (A)>1\} \,.
\]
Then it follows that, for each $U_1, U_2 \in W_2$,
\begin{equation}\label{e2}
\mathsf k (U_1)\ge \max \{3, n-2\} \quad \text{and} \quad \big(\text{either} \ \mathsf k (U_1)=\mathsf k (U_2) \ \text{or} \ |\mathsf k (U_1)-\mathsf k (U_2)|\ge  \max \{2, n-3\} \big) \,.
\end{equation}

We choose an element $U \in W_2$. Then $\supp (U) = G_0$, and we pick an element $g\in G_0\setminus G_2$. Then $g\in \langle G_2 \rangle$ and, by Lemma \ref{3.4}.1,  there is an atom $A$ with
$\mathsf v_g(A)=1$ and $\supp(A)\subset G_2 \cup \{g\}\subsetneq G_0$. This implies that  $A\in W_1$,
and $$U  A^{\ord(g)-\mathsf v_g(U)}=g^{\ord(g)}  S$$
for some zero-sum sequence $S$ over $G$. Since
$\supp(S)=G_0\setminus\{g\}$ and $G_0$ is minimal non-half-factorial, $S$ has a factorization  into a product of atoms from $W_1$.
Therefore,  for each $U\in W_2$,  there are $A_1,  \ldots , A_m \in W_1$,
where $m \le \ord(g)-\mathsf v_g(U)\le n-1$,  such that $UA_1 \cdot \ldots \cdot A_m$ can be factorized into a product of atoms from $W_1$.

We set
\[
W_0=\{A\in \mathcal A (G_0) \mid \mathsf k (A)= \min \{ \mathsf k (B) \mid B \in W_2\} \}\subset W_2 \,,
\]
 and we consider all tuples $(U,A_1,\ldots,A_m)$, where $U\in W_0$ and $A_1,\ldots,A_m\in W_1$,  such that $UA_1 \cdot \ldots \cdot A_m$ can be factorized into a product of atoms from $W_1$. We fix one such tuple $(U,A_1,\ldots,A_m)$ with the property that $m$ is minimal possible.
Note  that $m\le n-1$.
Let
\begin{equation}\label{e3}
UA_1 \cdot \ldots \cdot A_m = V_1 \cdot \ldots \cdot V_t \quad \text{ with} \quad t \in \N \quad \text{and} \quad  V_1, \ldots, V_t \in W_1 \,.
\end{equation}
We observe that  $\mathsf k(U)=t-m$ and continue with the following assertion.

\begin{enumerate}
\item[{\bf A1.}\,] For each $\nu \in [1, t]$, we have $V_{\nu} \nmid UA_1 \cdot \ldots \cdot A_{m-1}$.
\end{enumerate}

\smallskip
{\it Proof of \,{\bf A1}}.\ Assume to the contrary that there is such a $\nu \in [1, t]$, say $\nu = 1$, with  $V_1 \t U A_1 \cdot \ldots \cdot A_{m-1}$. Then there are $l \in \N$ and $T_1, \ldots, T_l \in \mathcal A (G_0)$ such that
\[
U A_1 \cdot \ldots \cdot A_{m-1} = V_1 T_1 \cdot \ldots \cdot T_{\mathit l} \,.
\]
By the minimality of $m$,  there exists some $\nu \in [1, l]$ such that $T_{\nu} \in W_2$, say $\nu=1$.  Since
\[
\sum_{\nu=2}^l \mathsf k (T_{\nu}) = \mathsf k (U) + (m-1)-1 - \mathsf k (T_1) \le m-2 \le n-3 \,,
\]
and $\mathsf k (T')\ge n-2$ for all $T' \in W_2$, it follows that  $T_2, \ldots, T_l \in W_1$, whence $l = 1 + \sum_{\nu=2}^l \mathsf k (T_{\nu}) \le m-1$.
We obtain that
\[
V_1 T_1 \cdot \ldots \cdot T_{\mathit l}A_m = U A_1 \cdot \ldots \cdot A_m = V_1 \cdot \ldots \cdot V_t \,,
\] and thus
 \[
 T_1 \cdot \ldots \cdot T_{\mathit l}A_m = V_2 \cdot \ldots \cdot V_t \,.
\]
The minimality of $m$ implies that $\mathsf k(T_1)> \mathsf k(U)$.
It follows that  \[\mathsf k(T_1)-\mathsf k(U)=m-1-{\mathit l}\le m-2\le n-3\le \max\{n-3,2\}\le \mathsf k(T_1)-\mathsf k(U).\] Therefore  $l=1$, $m=n-1$,  $n\ge 5$ and $\mathsf k(T_1)=\mathsf k(U)+ n-3$.
Thus
\[
T_1A_{n-1}=V_2 \cdot \ldots \cdot V_t  \,, \quad \text{and hence} \quad t-1\le |A_{n-1}| \,.
\]
This equation shows  that $\mathsf k(T_1)=t-2\le |A_{n-1}|-1\le n-1$, and hence $n-2\le \mathsf k(U)=\mathsf k(T_1)-n+3\le 2$,
 a contradiction to $n\ge 5$. \qed{(Proof of {\bf A1})}

\medskip

Since $\exp (G)=n$ and $\mathsf k (A_m)=1$, it follows that $|A_m| \le n$.
By {\bf A1}, for each $\nu \in [1, t]$ there exists an element $h_{\nu} \in \supp (A_m)$ such that
\[
\mathsf v_{h_{\nu}} (V_{\nu}) > \mathsf v_{h_{\nu}} (U A_1 \cdot \ldots \cdot A_{m-1}) \,.
\]
For each $h \in \supp (A_m)$ we define
\[
F_h = \{ \nu \in [1,t] \ \mid \ \mathsf v_{h} (V_{\nu}) > \mathsf v_{h} (U A_1 \cdot \ldots \cdot A_{m-1}) \} \subset [1,t] \,.
\]
Thus
\[
\bigcup_{h\in \supp(A_m)}F_h=[1,t]\, \]
and for each $h \in \supp (A_m)$, we have \[   \mathsf v_h (A_m) + \mathsf v_h (UA_1 \cdot \ldots \cdot A_{m-1})=\sum_{i=1}^{t}\mathsf v_h(V_i)\ge \sum_{i\in F_h} \mathsf v_h(V_i)\ge |F_h|\big(\mathsf v_h (UA_1 \cdot \ldots \cdot A_{m-1})+1\big)  \,.
\]
Since $|A_m|>|\supp(A_m)|$ (otherwise, it would follow that  $A_m \t U$, a contradiction), we obtain that
\[
\begin{aligned}
t & = \Big| \bigcup_{h \in \supp (A_m)} F_h \Big| \le \sum_{h} |F_h|
  \le \sum_{h} \frac{\mathsf v_h (A_m) + \mathsf v_h (UA_1 \cdot \ldots \cdot A_{m-1})}{\mathsf v_h (UA_1 \cdot \ldots \cdot A_{m-1})+1} \\
  & \le \sum_{h} \frac{\mathsf v_h (A_m)+1}{2} = \frac{|A_m|}{2}+\frac{|\supp(A_m)|}{2} < |A_m| \le n \,.
\end{aligned}
\]

By Equations \eqref{e3} and \eqref{e2}, we have $\max\{3,n-2\}\le \mathsf k(U)=t-m\le n-1-m$ and hence $m=1$, $n\ge 5$, $t=n-1$, and $\mathsf k(U)=n-2$. Therefore
\begin{equation}\label{e4}UA_1 = V_1 \cdot \ldots \cdot V_{n-1},\ \,  |A_1|=n,\ \, n-2\le|\supp(A_1)|\le n-1 \ \,,
\end{equation}
and
\begin{equation}\label{e6} \sum_{h\in \supp(A_1)} |F_h|=n-1,\quad  \text{and the sets $F_h, h\in \supp(A_1)$ are pairwise disjoint.}
\end{equation}
Furthermore, $ |F_h|\le\frac{\mathsf v_h (A_1) + \mathsf v_h (U)}{\mathsf v_h (U)+1}$ for each $h\in \supp(A_1)$. Then for each $h\in \supp(A_1)$, we have that
\begin{equation}\label{e5}
|F_h|\le 1 \quad\text{ when } \mathsf v_h(A_1)\le 2\, \quad\text{ and }\,\quad  |F_h|\le 2 \quad \text{ when } \mathsf v_h(A_1)\le 4\,.
\end{equation}

\smallskip
Now we consider all atoms $A_1\in W_1$ such that $U A_1$ can be factorized into a product of $n-1$ atoms from $W_1$, and among them  the atoms $A_1'$ for which  $|\supp(A_1')|$ is minimal, and among them we choose an atom $A_1''$ for which   $\mathsf h(A_1'')$ is minimal. Changing notation if necessary we suppose that $A_1$ has this property. By Equation \eqref{e4}, we distinguish three cases depending on $|\supp(A_1)|$ and $\mathsf h(A_1)$.

\smallskip
\noindent
CASE 1: \  $|\supp(A_1)|= n-1$.

Let $\supp(A_1)=\{g_1,\ldots,g_{n-1}\}$ and $A_1=g_1^2g_2\cdot \ldots \cdot g_{n-1}$. Since $\mathsf h(A_1)=2$,  Equations \eqref{e5} and \eqref{e6} imply that  $|F_h|=1$  for each $h\in \supp(A_1)$.
Note that $Ug_1^2g_2 \cdot \ldots \cdot  g_{n-1}=V_1 \cdot \ldots \cdot V_{n-1}$. After renumbering if necessary we may suppose that $F_{g_i} = \{i\}$ for each $i \in [1, n-1]$. Therefore, we have $\mathsf v_{g_i} (V_i) > \mathsf v_{g_i} (U) \ge 1$ for each $i \in [1, n-1]$. Hence $\mathsf v_{g_1}(V_1) \ge 2$ and we set $V_1=g_1^2Y_1$ for some $Y_1$ dividing $U$. Thus $UY_1^{-1}g_2 \cdot \ldots \cdot g_{n-1} = V_2 \cdot \ldots \cdot V_{n-1}$ which  implies that $V_i=g_iY_i$, for $i \in [2,n-1]$,  where $Y_2 \cdot \ldots \cdot Y_{n-1}=UY_1^{-1}$.
Summing up we have
\begin{equation}\label{e7}
U=Y_1\cdot\ldots\cdot Y_{n-1} \text{ such that }
V_i= g_iY_i \, \text{ for }\, i\in [2,\, n-1]\, \text{ and }\, V_{1}= g_1^2Y_1.
\end{equation}

If $n$ is even and $X \in \mathcal A (G)$ such that $X \t A_1^{n/2}$, then $\mathsf k (X) \le (n/2) \mathsf k (A_1) = n/2 < n-2$ whence $X \in W_1$ and  $\mathsf k (X)=1$. This shows that  $\mathsf L (A_1^{n/2})= \{n/2\}$.
Similarly, if $n$ is odd, then $\mathsf L (A_1^{(n+1)/2}) = \{(n+1)/2\}$.
Therefore,\[A'=\left\{
\begin{aligned}
A_1^{\frac{n}{2}}&=g_1^n  g_2^{\frac{n}{2}}\cdot \ldots \cdot g_{n-1}^{\frac{n}{2}} \quad \text{can only be written as a product of $n/2$ atoms if $n$ is even, }\\
A_1^{\frac{n+1}{2}}&=g_1^n g_1g_2^{\frac{n+1}{2}}\cdot \ldots \cdot g_{n-1}^{\frac{n+1}{2}} \quad \text{can only be written as product of $(n+1)/2$ atoms if $n$ is odd}\,.
\end{aligned}\right.
\]
Thus we can find an atom $C\t A'(g_1^n)^{-1}$ with $\supp(C)\subset \{g_2,\ldots,g_{n-1}\}$ and $|\supp(C)|\ge 2$, say $g_2,g_3\in \supp(C)$.  Therefore, we obtain that $V_2V_3=g_2g_3Y_2Y_3\t UC$, say $UC = V_2 V_3V'$ for some $V' \in \mathcal B (G)$. Since
 \[
 \mathsf k (UC)=\mathsf k (U) + \mathsf k (C) = n-1=\mathsf k (V_2)+\mathsf k (V_3)+\mathsf k (V') \,,
 \]
 we obtain that $\mathsf k (V')=n-3$. Now Equation \eqref{e2} implies that $V'$ is a product of atoms from $W_1$,
 and hence $UC$ can be factorized into a product of $n-1$ atoms. Since $|\supp(C)|<n-1=|\supp(A_1)|$, this is a contradiction to the choice of $A_1$.

\smallskip
\noindent
CASE 2: \  $|\supp(A_1)|= n-2$ and  $\mathsf h(A_1)=2$.

Let $\supp(A_1)=\{g_1,\ldots,g_{n-2}\}$ and $A_1=g_1^2g_2^2g_3\cdot \ldots \cdot g_{n-2}$. Since $\mathsf h(A_1)=2$,  Equation \eqref{e5} implies that  $|F_h|\le 1$ for each $h\in \supp(A_1)$. Thus $\sum_{h\in \supp(A_1)} |F_h|\le n-2$, a contradiction to Equation \eqref{e6}.

\smallskip
\noindent
CASE 3: \ $|\supp(A_1)|= n-2$ and  $\mathsf h(A_1)=3$.

Let $\supp(A_1)=\{g_1,\ldots,g_{n-2}\}$ and $A_1=g_1^3g_2 \cdot \ldots \cdot g_{n-2}$. Since $\mathsf h(A_1)=3$,  the Equations \eqref{e5} and \eqref{e6} imply that  $|F_{g_1}|=2$ and $|F_{g_i}|=1$ for each $i\in [2,\,n-2]$.
Note that $Ug_1^3g_2 \cdot \ldots \cdot g_{n-2}=V_1\cdot \ldots \cdot V_{n-1}$. After renumbering if necessary we may suppose that $F_{g_1}=\{1, n-1\}$ and $F_{g_i}=\{i\}$ for each $i \in [2, n-2]$. Therefore we have $\mathsf v_{g_i} (V_i) > \mathsf v_{g_i}(U) \ge 1$ for each $i \in [1, n-2]$ and $\mathsf v_{g_1}(V_{n-1})>\mathsf v_{g_1}(U)\ge 1$. Hence we may set $V_{n-1}=g_1^2Y_{n-1}$ for some $Y_{n-1}$ dividing $U$. Thus $UY_{n-1}^{-1}g_1g_2 \cdot \ldots \cdot g_{n-2}=V_1 \cdot \ldots \cdot V_{n-2}$ which implies that $V_i=g_iY_i$ for each $i \in [1, n-2]$ where $Y_1 \cdot \ldots \cdot Y_{n-2} = UY_{n-1}^{-1}$. Summing up we have
\begin{equation}\label{e8}
U=Y_1\cdot\ldots\cdot Y_{n-1} \text{ such that } V_i= g_iY_i \, \text{ for }\, i\in [1,\, n-2]\, \text{ and }\, V_{n-1}= g_1^2Y_{n-1}.
\end{equation}

 As in CASE 1 we obtain that (note $n\ge 5$)
 \[A'=\left\{
 \begin{aligned}
A_1^{\frac{n}{3}}&=g_1^n  g_2^{\frac{n}{3}}\cdot \ldots \cdot  g_{n-2}^{\frac{n}{3}} \quad\ \ \ \, \text{can only be written as a product of $\,\frac{n}{3}$ atoms if} \   n\equiv 0\mod{3}\\
A_1^{\frac{n+1}{3}}&=g_1^n  g_1g_2^{\frac{n+1}{3}}\cdot \ldots \cdot g_{n-2}^{\frac{n+1}{3}} \  \text{can only be written as a product of $\,\frac{n+1}{3}$ atoms if} \   n\equiv 2\mod{3}\\
A_1^{\frac{n+2}{3}}&=g_1^n  g_1^2 g_2^{\frac{n+2}{3}} \cdot \ldots \cdot g_{n-2}^{\frac{n+2}{3}} \  \text{can only be written as a product of $\,\frac{n+2}{3}$ atoms if}\  n\equiv 1\mod{3}\,.
\end{aligned}\right.\]
Let $C \in \mathcal A (G)$ be an atom dividing $A'(g_1^n)^{-1}$. Then $\supp(C)\subset \{g_1,\ldots,g_{n-2}\}$ and $|\supp(C)|\ge 2$, say $g_i,g_j\in \supp(C)$ where $1\le i<j\le n-2$. Therefore, we obtain that $V_iV_j=g_ig_jY_iY_j\t UC$ by Equation \eqref{e8}. Arguing as in CASE 1 we infer that  $UC$ is a product of $n-1$ atoms from $W_1$.
By the choice of $A_1$, we obtain that $|\supp(C)|=n-2$ and $\mathsf h(C)\ge 3$. Since this holds for all atoms dividing $A'(g_1^n)^{-1}$, we obtain  a contradiction to the structure of $A'$.
\end{proof}

\medskip
\begin{proof}[Proof of Theorem \ref{1.1}]
Let $H$ be a Krull monoid with class group $G$ and let $G_P \subset G$ denote the set of classes containing prime divisors.
If $|G| \le 2$, then $H$ is half-factorial by \cite[Corollary 3.4.12]{Ge-HK06a}, and thus $\Delta^* (H) \subset \Delta (H) = \emptyset$.
If $G$ is infinite and $G_P =G$, then $\Delta^* (H) = \mathbb N$ by \cite[Theorem 1.1]{Ch-Sc-Sm08b}.

Suppose that  $2 < |G| < \infty$. By Lemma \ref{2.1}, it suffices to prove the statements for the Krull monoid $\mathcal B (G_P)$.
If $G$ is finite, then $\Delta (G)$ is finite by \cite[Corollary
3.4.13]{Ge-HK06a}, hence $\Delta^* (G)$ is finite, and
Lemma \ref{3.1} shows that $\{ \exp (G)-2, \mathsf r (G)-1\} \subset \Delta^* (G)$.

Since $\Delta^* (G_P) \subset \Delta^* (G)$, it remains to prove that
\[
\max \Delta^* (G) \le \max \{ \exp (G)-2, \mathsf r (G)-1\} \,.
\]

Let $G_0 \subset G$ be a  non-half-factorial subset, $n=\exp(G)$, and $r=\mathsf r(G)$.  We  need to prove that $\min \Delta (G_0) \le \max \{ n-2, r-1\}$.
If $G_1 \subset G_0$ is non-half-factorial, then $\min \Delta (G_0) = \gcd \Delta (G_0) \t \gcd \Delta (G_1) = \min \Delta (G_1)$. Thus we may suppose that $G_0$ is minimal non-half-factorial.
If there is an $U \in \mathcal A (G_0)$ with $\mathsf k (U)<1$, then Lemma \ref{3.1}.3 implies that $\min \Delta (G_0) \le n-2$.  Suppose that $\mathsf k (U) \ge 1$ for all $U \in \mathcal A (G_0)$, i.e, $G_0$ is an LCN-set.
Since $G_0$ is minimal non-half-factorial, it follows that $G_0$ is indecomposable by Lemma \ref{3.3}. By Lemma \ref{3.4}.3, we may suppose that for each $g \in G_0$ we have $g \in \langle G_0 \setminus \{g\} \rangle$.
Suppose that the order of each element of $G_0$ is a prime power. Since $G_0$ is indecomposable, Lemma \ref{3.3} implies that each order is a power of a fixed prime $p \in \mathbb P$, and thus $\langle G_0 \rangle$ is a $p$-group. By Proposition \ref{3.2} we infer that
\[
\min \Delta (G_0) \le \max \Delta^* ( \langle G_0 \rangle) = \max \{ \exp ( \langle G_0 \rangle ) - 2, \mathsf r ( \langle G_0 \rangle)-1 \} \le \max \{ n-2, r-1 \} \,.
\]
From now on we suppose that there is an element $g \in G_0$ whose order is not a prime power. Then $n \ge 6$.
If $|G_0| \le r+1$, then $\min \Delta (G_0) \le |G_0|-2\le r-1$ by Lemma \ref{3.1}.3. Thus we may suppose that  $|G_0| \ge r+2$ and we distinguish two cases.

\smallskip
\noindent
CASE 1: There exists a subset $G_2 \subset G_0$ such that $\langle G_2 \rangle = \langle G_0 \rangle$ and $|G_2| \le |G_0|-2$.

Then Lemma \ref{3.6} implies that $\min \Delta (G_0) \le  n-4 \le n-2$.

\smallskip
\noindent
CASE 2: Every subset $G_1\subset G_0$ with $|G_1|=|G_0|-1$ is a minimal generating set of $\langle G_0 \rangle $.

Then for each $h\in G_0$, $G_0\setminus \{h\}$ is half-factorial and $h \notin \langle G_0 \setminus \{h, h' \} \rangle \ \text{for any }\ h'\in G_0\setminus\{h\}$. Thus Lemma \ref{3.5} implies that $|G_0|\le n-2 $ and hence $\min \Delta (G_0) \le |G_0|-2\le n-4 \le n-2$ by Lemma \ref{3.1}.3.
\end{proof}

\medskip
\section{Inverse results on $\Delta^* (H)$} \label{4}
\medskip

Let $G$ be a finite abelian group. In this section we study the structure of minimal non-half-factorial subsets $G_0 \subset G$ with $\min \Delta (G_0) = \max \Delta^* (G)$.
These structural investigations were started by Schmid who obtained  a characterization in case $\exp (G)-2 > \mathsf m (G)$ (Lemma \ref{4.1}.1). Our main result in this section is Theorem \ref{4.5}. All examples of minimal non-half-factorial subsets $G_0 \subset G$ with $\min \Delta (G_0) = \max \Delta^* (G)$ known so far are simple, and the standing conjecture was that all such sets are simple. We provide the  first example of such a set $G_0$ which is not simple (Remark \ref{4.6}).

\medskip
\begin{lemma} \label{4.1}
Let $G$ be a finite abelian group with $|G| > 2$, $\exp (G)=n$, $\mathsf r (G)=r$,  and let $G_0 \subset G$ be a subset with $\min \Delta (G_0) = \max \Delta^* (G)$.
\begin{enumerate}
\item Suppose that $\mathsf m (G) < n-2$. Then $G_0$ is indecomposable if and only $G_{0}=\{g, -g\}$ for some $g\in G$ with $\ord (g) = n$.

\item Suppose that $r \le n-1$. Then  $G_0$ is minimal non-half-factorial but not an {\rm LCN}-set if and only if  $G_{0}=\{g, -g\}$ for some $g\in G$ with $\ord (g) = n$.
\end{enumerate}
\end{lemma}

\begin{proof}
1. See \cite[Theorem 5.1]{Sc08d}.

\smallskip
2. Since $n=2$ implies $r=1$ and $|G|=2$, it follows that $n \ge 3$. By Theorem \ref{1.1}, we have that $\min \Delta(G_0)=n-2$. Obviously, the set $\{-g, g\}$, with $g \in G$ and $\ord (g)=n$, is a minimal non-half-factorial set with $\min \Delta (\{-g,g\})=n-2$ but not an LCN-set. Conversely, let $G_0$ be  minimal non-half-factorial but not an {\rm LCN}-set.
Then there exists an  $A \in \mathcal A (G_0)$ with $\mathsf k(A)<1$.  Since $\{n, n\mathsf k(A^n)\}\subset \mathsf L(A^n)$, it follows that  $n-2 \t n(\mathsf k(A)-1)$ whence $\mathsf k(A)=\frac{2}{n}$. Consequently, $A=(-g)g$ for some $g$ with $\ord(g)=n$. Thus $\{-g, g\}\subset G_0$, and since $G_0$ is minimal non-half-factorial, equality follows.
\end{proof}

\medskip
\begin{lemma} \label{4.2}
Let $G$ be a finite abelian group  with $\exp(G)=n$, $\mathsf r (G)=r$, and let $G_0 \subset G$ be a minimal non-half-factorial {\rm LCN}-set with $\min  \Delta (G_0)=\max \Delta^* (G) $.
\begin{enumerate}
\item  Then $|G_0|=r+1$, $r\ge n-1$ and for each $h\in G_0$,  $h\not\in \langle G_0\setminus \{h,\ h'\}\rangle $ for any $h'\in G_0\setminus \{h\}$.

\smallskip
\item If  $r \le n-2$, then $\mathsf m(G) \le n-3$.

\smallskip
\item If $n \ge 5$ and $r \le n-3$ then $\mathsf m (G) \le n-4$.
\end{enumerate}
\end{lemma}

\begin{proof}
1. We have that $\min  \Delta (G_0)\le |G_0|-2$ by Lemma \ref{3.1}.3 and $\min  \Delta (G_0)= \max \{n-2, r-1\}$ by Theorem \ref{1.1}.

By Lemma \ref{3.4}.3 (Properties (a) and (c)), we may assume that for each $g \in G_0$ we have $g \in \langle G_0 \setminus \{g\} \rangle$.

\smallskip
\noindent
CASE 1:  There is a
subset $G_2 \subset G_0$ such that $\langle G_2 \rangle = \langle G_0 \rangle$ and $|G_2| \le |G_0|-2$.

The existence of $G_2$ implies that $G$ is neither isomorphic to $C_3$ nor to $C_2 \oplus C_2$ nor to $C_3 \oplus C_3$ (this is immediately clear for the first two groups; to exclude the case $C_3 \oplus C_3$, use again \cite[Corollary 3.1]{Sc05d} which  says that $\mathsf k (U) \in \N$ for each $U \in \mathcal A (G_0)$).
By Lemma \ref{3.6}, we know that $\min  \Delta (G_0) \le \max \{n-4, 1\}  < \max \{n-2, r-1\} = \min \Delta (G_0)$, a contradiction.

\smallskip
\noindent
CASE 2: Every subset $G_1\subset G_0$ with $|G_1|=|G_0|-1$ is a minimal generating set of $\langle G_0 \rangle $.

Then for each $h\in G_0$,  $h\not\in \langle G_0\setminus \{h,\ h'\}\rangle $ for any $h'\in G_0\setminus \{h\}$.

If $|G_0|\ge r+2$, then by Lemma \ref{3.5} $|G_0|\le n-2$, it follows that $\min  \Delta (G_0)\le |G_0|-2 \le n-4$, a contradiction.

If $|G_0|\le r+1$, then $\max \{n-2, r-1\}=\min  \Delta (G_0)\le |G_0|-2 \le r-1$,
so we must have $|G_0|=r+1$ and $r\ge n-1$.

\smallskip
2.
Assume to the contrary that $r \le n-2$ and that  $\mathsf m(G) \ge n-2$. Then by Theorem \ref{1.1}, $\max \Delta^* (G) = \max \{r-1, n-2\} = n-2$. Since $\mathsf m(G) \ge n-2$, there is  a minimal non-half-factorial {\rm LCN}-set $G_0$ with $\min  \Delta (G_0)=\max \Delta^* (G) $, and then 1. implies that $r \ge n-1$, a contradiction.

\smallskip
3.
Let $G_0 \subset G$ be a  non-half-factorial LCN-subset.  We  need to prove that $\min \Delta (G_0) \le  n-4$. Without restriction we may
suppose that $G_0$ is minimal non-half-factorial which implies that $G_0$ is indecomposable by Lemma \ref{3.3}. By Lemma \ref{3.4}.3, we may suppose that for each $g \in G_0$ we have $g \in \langle G_0 \setminus \{g\} \rangle$.
Suppose that the order of each element of $G_0$ is a prime power. Since $G_0$ is indecomposable, Lemma \ref{3.3} implies that each order is a power of a fixed prime $p \in \mathbb P$, and thus $\langle G_0 \rangle$ is a $p$-group. By Proposition \ref{3.2}, we infer that
\[
\min \Delta (G_0) \le \mathsf m ( \langle G_0 \rangle) =  \mathsf r ( \langle G_0 \rangle)-1  \le  \mathsf r (G)-1  \le n-4 \,.
\]
From now on we suppose that there is an element $g \in G_0$ whose order is not a prime power.
If $|G_0| \le n-2$, then $\min \Delta (G_0) \le |G_0|-2\le  n-4$ by Lemma \ref{3.1}.3. Thus we may suppose that  $|G_0| \ge n-1\ge r+2$ and we distinguish two cases.

\smallskip
\noindent
CASE 1: There exists a subset $G_2 \subset G_0$ such that $\langle G_2 \rangle = \langle G_0 \rangle$ and $|G_2| \le |G_0|-2$.

Then Lemma \ref{3.6} implies that $\min \Delta (G_0) \le  n-4 $.

\smallskip
\noindent
CASE 2: Every subset $G_1\subset G_0$ with $|G_1|=|G_0|-1$ is a minimal generating set of $\langle G_0 \rangle $.

Then for each $h\in G_0$, $G_0\setminus \{h\}$ is half-factorial and $h \notin \langle G_0 \setminus \{h, h' \} \rangle \ \text{for any }\ h'\in G_0\setminus\{h\}$. Thus Lemma \ref{3.5} implies that $|G_0|\le n-2$, a contradiction.
\end{proof}

\medskip
\begin{lemma} \label{4.3}
Let $G$ be a finite abelian group  with $\exp(G)=n$, $\mathsf r (G)=r$, and let $G_0 \subset G$ be a minimal non-half-factorial {\rm LCN}-set with $\min  \Delta (G_0)=\max \Delta^* (G) $.
\begin{enumerate}
\item If $A\in \mathcal A(G_0)$ with $\mathsf k(A)=1$, then $|\supp(A)|\le\frac{n}{2}$.

\smallskip
\item  If $A\in \mathcal A(G_0)$ with $\mathsf k(A)>1$, then $\mathsf k(A)<r$ and $SA^{-1}$ is also an atom where $S=\prod_{g\in G_0}g^{\ord(g)}$.
\end{enumerate}
\end{lemma}

\begin{proof}
By Lemma \ref{4.2}, we have $r \ge n-1$, $|G_0|=r+1$, and for each $h\in G_0$,  $h\not\in \langle G_0\setminus \{h,\ h'\}\rangle $ for any $h'\in G_0\setminus \{h\}$. Let $A \in \mathcal A (G_0)$.

\smallskip
1.   Since $\mathsf k (A)=1$, it follows that $|\supp (A)| \le |A| \le n$. Assume  that $|\supp (A)| = n$. Then $\mathsf v_g (A)=1$ for each $g \in \supp (A)$. Since $G_0$ is a minimal non-half-factorial {\rm LCN}-set, there is a $V \in \mathcal A (G_0)$ with $\mathsf k (V) > 1$ and $\supp (V) = G_0$. Therefore $A \t V$, a contradiction. Thus $|\supp (A)| \le n-1$ whence $\supp (A) \subsetneq G_0$.
  Therefore Lemma \ref{3.4}.2 implies  that $\gcd(\mathsf v_g(A), \ord(g))>1$ for each $g\in \supp(A)$, and hence $|\supp(A)| \le |A|/2 \le n/2$.

\smallskip
2. Let $A \in \mathcal A (G_0)$ with $\mathsf k (A)>1$. Then $A \t S$, $r+1 = |G_0| = \max  \mathsf L (S)$, and $\mathsf L (S) \setminus \{r+1\} \ne \emptyset$. By Theorem \ref{1.1}, we have  $\min \Delta (G_0)=r-1$, hence  $\mathsf L(S)=\{2,r+1\}$, and thus
$SA^{-1}$ is  an atom. If $\mathsf k (SA^{-1})=1$, then 1. implies that $|\supp (SA^{-1}) | \le n/2$, but on the other hand we have $|\supp (SA^{-1})| = |G_0| = r+1 \ge n$, a contradiction.
Therefore we obtain that $\mathsf k(SA^{-1})>1$ and hence
$r+1= \mathsf k (S) = \mathsf k (A) + \mathsf k (SA^{-1})$ implies that
 $\mathsf k(A)<r$.
\end{proof}

\medskip
\begin{lemma} \label{4.4}
Let $G$ be a finite abelian group  with $\exp(G)=n$, $\mathsf r (G)=r$, and let $G_0 \subset G$ be a minimal non-half-factorial {\rm LCN}-set with $\min  \Delta (G_0)=\max \Delta^* (G) $. Let  $g\in G_0$ with $g \in \langle G_0 \setminus \{g\}\rangle$ and $d \in [1, \ord (g)]$ be  minimal  such that $dg\in \langle E^* \rangle$ for some subset $E^* \subsetneq G_0\setminus \{g\}$.  Then  $d \t \ord(g)$, and we have
\begin{enumerate}
\item Let  $k\in[1,\,\ord(g)]$. If $kg \not\in \langle E \rangle$ for any $E\subsetneq G_0\setminus \{g\}$, then there is an atom $A$ with $\mathsf v_g(A)=k$ and $\mathsf k (A)>1$.

\smallskip
\item  Let $k\in[1,\,\ord(g)-1]$ with $d\nmid k$. Then there is an atom $A$ with $\mathsf v_g(A)=k$ and $\mathsf k(A)>1$. In particular, if   $B\in \mathcal B(G_0)$ with $\mathsf v_g(B)=k$ and $B \t \prod_{g\in G_0}g^{\ord(g)}$, then $B$ is an atom.

\smallskip
\item  If $A_1,A_2$ are atoms with $\mathsf v_g(A_1) \equiv \mathsf v_g(A_2)\mod d$, then $\mathsf k(A_1)=\mathsf k(A_2)$.
\end{enumerate}
\end{lemma}

\begin{proof}
Note that by Lemma \ref{4.2}, we have $|G_0|=r+1$ and $r\ge n-1$. The minimality of $d$ and Lemma \ref{3.4}.1 imply that  $d \t \ord(g)$. We set $S=\prod_{g\in G_0}g^{\ord(g)}$.

\smallskip
1. Since $kg\in \langle G_0\setminus\{g\}\rangle$, there is a zero-sum sequence $A$  such that $\mathsf v_g(A)=k$, and we choose an $A$ with minimal length $|A|$.  Then $\supp(A)=G_0$ by assumption on $kg$, and we assert  that $A$ is an atom. If this holds, then $\mathsf k (A)>1$ by Lemma \ref{4.3}.1.

Assume to the contrary that $A=A_1 \cdot \ldots \cdot A_s$ with $s\ge2$ and  atoms $A_1, \ldots, A_s$. The minimality of $|A|$ implies that $\mathsf v_g (A_i)>0$ for each $i \in [1,s]$.
If there exists an $i \in [1,s]$ such that $\mathsf k(A_i)>1$, say $A_1$, then $S=A_1 \cdot \ldots \cdot A_s(SA^{-1})$ but $SA_1^{-1}=A_2 \cdot \ldots \cdot A_s(SA^{-1})$ is not an atom, a contradiction to Lemma \ref{4.3}.2. Thus, for each $i \in [1,s]$, we have  $\mathsf k(A_i)=1$ and hence $\supp(A_i)\subsetneq G_0$ by Lemma \ref{4.3}.1.

For each $i\in [1,\,s]$, we set $t_i = \mathsf v_g (A_i)$,
$d_i=\gcd( \{t_1, \ldots, t_i,\ord(g) \})$, and let  $E_i \subset G_0 \setminus \{g\}$ be minimal such that $d_ig\in \langle E_i \rangle$. Note that $k=t_1+ \ldots + k_s$.
Since $d_1g\in \langle t_1g \rangle \subset \langle \supp(A_1)\setminus\{g\} \rangle \subsetneq \langle G_0\setminus\{g\} \rangle$, it follows that $E_1\subsetneq G_0\setminus\{g\}$.
Since
$kg\in \langle d_sg \rangle \subset \langle E_s\rangle $, it follows that $E_s=G_0\setminus\{g\}$.

Let $l \in [1, s-1]$ be maximal  such that $E_l \subsetneq G_0\setminus\{g\}$.
Then $d_lg\in \langle E_l\rangle$ and $E_{l+1}=G_0\setminus\{g\}$.
Let $d_0 \in \N$ be the minimal  such that $d_0g\in E_l$. Then Lemma \ref{3.4}.1 implies  that  $d_0 \t d_l$
 and there exists an atom $W$ such that $\supp(W)=\{g\}\cup E_l$, $\mathsf v_g(W)=d_0$, and $\mathsf k(W)=1$.
Since $d_{l+1}g\in \langle d_lg, t_{l+1}g \rangle\subset  \langle E_l\cup \supp(A_{l+1})\setminus\{g\} \rangle $, we have that $ E_l\cup \supp(A_{l+1})\setminus\{g\}=G_0\setminus\{g\}$.
Then Lemma \ref{4.3}.1 implies that $|G_0|\le 1+ |E_l| + |\supp (A_{l+1}) \setminus \{g\}| \le 1 + (n/2-1)+(n/2-1) = n-1$, a contradiction.

\smallskip
2. If $kg \in \langle E_1 \rangle$ for some $E_1 \subsetneq G_0 \setminus \{g\}$, then $\gcd (d,k)g \in \langle k g \rangle \subset \langle E_1 \rangle$, whence the minimality of $d$ implies that $\gcd (d,k)=d$ and $d \t k$, a contradiction. Therefore, we obtain that $kg\not\in \langle E \rangle$ for any $E\subsetneq G_0\setminus \{g\}$. Thus 1. implies that  there is an atom $A$ with $\mathsf v_g(A)=k$ and $\mathsf k(A)>1$.

Let $B \in \mathcal B (G_0)$ with $B \t S$ and $\mathsf v_g (B)=k$. We set $B = A_1 \cdot \ldots \cdot A_s$ with $s \in \N$ and atoms $A_1, \ldots, A_s$. Then $\mathsf v_g (A_1)+ \ldots + \mathsf v_g (A_s) = \mathsf v_g (B)=k$. Since $d \nmid k$, there is an $i \in [1,s]$ with $d \nmid \mathsf v_g (A_i)$. We want to show that $\mathsf k (A_i)>1$, and assume to the contrary that $\mathsf k (A_i)=1$. Then $|\supp (A_i)|\le n/2$ by Lemma \ref{4.3}.1. Furthermore, $d' = \gcd (d, \mathsf v_g (A_i) ) < d$, but
\[
d'g \in \langle \mathsf v_g (A_i)g \rangle \subset \langle \supp (A_i) \setminus \{g\} \rangle \quad \text{and} \quad \supp (A_i) \setminus \{g\} \subsetneq G_0 \setminus \{g\} \,,
\]
a contradiction to the minimality of $d$. Therefore it follows that $\mathsf k (A_i) > 1$. Since $g \t SB^{-1}$, it follows that $S \ne B$. Since $S = A_i \big( (BA_i^{-1}) (SB^{-1}) \big)$ and $SA_i^{-1}$ is an atom by Lemma \ref{4.3}.2, it follows that $B=A_i \in \mathcal A (G_0)$.

\smallskip
 3. Let $A_1 \in \mathcal A (G_0)$. We assert that $\mathsf k (A_1)=\mathsf k (A_2)$ for all $A_2 \in \mathcal A (G_0)$  with $\mathsf v_g(A_1) \equiv \mathsf v_g(A_2)\mod d$. We distinguish two cases.

\smallskip
\noindent
CASE 1: \    $d \t \mathsf v_g(A_1)$.

There is an $A \in \mathcal A (G_0)$ with $\mathsf v_g (A)=d$ and $\mathsf k (A)=1$. It is sufficient to show that $\mathsf k (A_1)=1$. There are $l \in \N$ and  $V_1, \ldots, V_l \in \mathcal A (G_0 \setminus \{g\})$ (hence $\mathsf k (V_1)= \ldots = \mathsf k (V_l)=1$) such that
\[
A_1 A^{\frac{\ord (g)- \mathsf v_g (A_1)}{d}} = g^{\ord (g)} V_1 \cdot \ldots \cdot V_l \quad \text{hence} \quad \mathsf k (A_1) = 1+l - \frac{\ord (g)- \mathsf v_g (A_1)}{d} \,.
\]
Furthermore, $\min \Delta (G_0)=r-1$ divides
\[
(l+1) - \Big( 1 + \frac{\ord (g)- \mathsf v_g (A_1)}{d} \Big) = \mathsf k (A_1)-1 \,.
\]
Since $\mathsf k (A_1)< r$ by Lemma \ref{4.3}, it follows that $\mathsf k (A_1)=1$.

\smallskip
\noindent
CASE 2: \  $d \nmid \mathsf v_g(A_1)$.

Let $d_0 \in [1,d-1]$ such that $\mathsf v_g (A_1) \equiv d_0 \mod d$. By 2., there are atoms $B_l$ such that $\mathsf v_g (B_l) = d_0 + ld$ for all $l \in \N_0$ with $d_0 + ld < \ord (g)$. Thus by an inductive argument it is sufficient to prove the assertion for those atoms $A_2$ with $\mathsf v_g(A_2)=\mathsf v_g(A_1)$ and with $\mathsf v_g(A_2)=\mathsf v_g(A_1)+d$.

Suppose that  $\mathsf v_g(A_1)=\mathsf v_g(A_2)$. By 2., there is an atom $V$ such that $\mathsf v_g(V)=\ord(g)-\mathsf v_g(A_1)$. Then there are $l \in \N$ and $V_1, \ldots, V_l \in \mathcal A (G_0 \setminus \{g\})$ such that $A_1V = g^{\ord (g)} V_1 \cdot \ldots \cdot V_l$ and hence $\mathsf k (A_1)+\mathsf k (V) = 1 + \sum_{i=1}^l \mathsf k (V_i) = l+1$. Since $\min \Delta (G_0) = r-1$ divides $l-1$, it follows that  either $l=r$ or $l \ge 2r-1$. If $l \ge 2r-1$, then $\mathsf k (A_1) \ge r$ or $\mathsf k (V) \ge r$, a contradiction to Lemma \ref{4.3}.
Therefore $\mathsf k(A_1)+\mathsf k(V)= r+1 = \mathsf k(A_2)+\mathsf k(V)$ and hence $\mathsf k(A_1)=\mathsf k(A_2)$.

Suppose that  $\mathsf v_g(A_1)=\mathsf v_g(A_2)+d$. Let $E \subsetneq G_0 \setminus \{g\}$ such that $dg \in \langle E \rangle$. Then there is an $A \in \mathcal A (E \cup \{g\})$  with $\mathsf v_g(A)=d$, and clearly $\mathsf k (A)=1$.  Let $V_1, \ldots,V_t$ be all the atoms with $V_{\nu} \t A_2A$ and $|\supp(V_{\nu})|=1$ for all $\nu \in[1,\,t]$. Since $\mathsf v_g (A_2A) = \mathsf v_g (A_1) < \ord (g)$, it follows that $B = A_2 A (V_1 \cdot \ldots \cdot V_t)^{-1}$ divides $S$ and that $\mathsf v_g (B)= \mathsf v_g (A_1)$. Therefore 2. implies that $B$ is an atom, and by Step 1 we obtain that $\mathsf k (B) = \mathsf k (A_1)$. If $t \ge 2$, then $A_2A=BV_1 \cdot \ldots \cdot V_t$ implies $t \ge 1+ \min \Delta (G_0) =r$, and thus $\mathsf k (A_2) \ge r$, a contradiction to Lemma \ref{4.3}. Therefore we obtain that $t=1$ and thus $\mathsf k (A_2)+1= \mathsf k (B)+1 = \mathsf k (A_1)+1$.
\end{proof}

\medskip
\begin{theorem} \label{4.5}
Let $G$ be a finite abelian group  with $\exp(G)=n$, $\mathsf r (G)=r$, and let $G_0 \subset G$ be a minimal non-half-factorial set with $\min  \Delta (G_0)=\max \Delta^* (G)$.
\begin{enumerate}
\item If $r < n-1$, then there exists $g\in G$ with $\ord (g)=n$ such that $G_0=\{g,\: -g\}$.

\smallskip
\item Let $r=n-1$. If $G_0$ is not an {\rm LCN}-set, then there exists $g\in G$ with $\ord (g)=n$ such that $G_0=\{g,\: -g\}$. If $G_0$ is  an {\rm LCN}-set, then $|G_0|=r+1$  and for each $h\in G_0$,  $h\not\in \langle G_0\setminus \{h,\ h'\}\rangle $ for any $h'\in G_0\setminus \{h\}$.

\smallskip
\item If $r\ge n $, then $G_0$ is an {\rm LCN}-set with $|G_0|=r+1$ and for each $h\in G_0$,  $h\not\in \langle G_0\setminus \{h,\ h'\}\rangle $ for any $h'\in G_0\setminus \{h\}$.

\smallskip
\item If  $r \ge n-1$, $G_0$ is an {\rm LCN}-set, and $n$ is odd, then there exists an element $g\in G_0$ such that $G_0\setminus \{g\} $ is independent.
\end{enumerate}
\end{theorem}

\begin{proof}
1. Suppose that $r < n-1$.  Then Lemma \ref{4.2} implies that $G_0$ is not an LCN-set. Thus   Lemma \ref{4.1}.2 implies that $G_0$ has the asserted form.

\smallskip
2. If  $G_0$ is not an {\rm LCN}-set, then the assertion follows from  Lemma \ref{4.1}.2. If $G_0$ is  an LCN-set, then the assertion follows from Lemma \ref{4.2}.1.

\smallskip
3. Suppose that $r\ge n$. Then Theorem \ref{1.1} implies that $\min \Delta (G_0)=\max \Delta^* (G)=r-1$.  Thus Lemma \ref{3.1}.3.(a) imply that $G_0$ is an LCN-set. Hence the assertion follows from Lemma \ref{4.2}.1.

\smallskip
4. Let $r \ge n-1$, $G_0$ be an {\rm LCN}-set, and suppose that $n$ is odd.
By Lemma \ref{3.4}.3 (Properties (a) and (d)), we may suppose without restriction that $g\in \langle G_0\setminus\{g\} \rangle$ for each $g \in G_0$. Lemma \ref{4.2} implies that $|G_0|=r+1$ and that for each $g \in G_0$ we have $g\not\in \langle E \rangle$ for any $E\subsetneq G_0\setminus \{g\}$.

Assume to the contrary that $G_0 \setminus \{h\}$ is dependent for each $h \in G_0$. Then there exist $g\in G_0$,  $d\in [2,\, \ord(g)-1]$, and $E\subsetneq G_0\setminus \{g\}$ such that  $dg\in \langle E \rangle$. Now let $d \in \N$  be  minimal over all  configurations $(g,E,d)$, and fix $g, E$ belonging to $d$. It follows that we have an atom $A$ with $\supp(A)\subsetneq G_0$ and $\mathsf v_g(A)=d$. By Lemma \ref{4.4}, we obtain that $d \t \ord (g)$, and hence $d\ge 3$ because $n$ is odd.

Since $G_0\setminus\{g\}$ is dependent, there  exist atoms $U' \in \mathcal A (G_0 \setminus \{g\})$ with $|\supp(U')|>1$. Thus, by Lemma \ref{3.4}.1, there exist an $U \in \mathcal A (G_0 \setminus \{g\})$ and an $h\in \supp(U)$ such that $\mathsf v_h(U)\le \frac{\ord(h)}{2}$ and $\mathsf v_h(U)\t \ord(h)$.

\vskip 2mm
By Lemma \ref{4.4}.2, there are atoms $A_1, \ldots,A_{d-1}$  with $\mathsf v_g(A_i)=i$ and $\mathsf k(A_i)>1$ for each $i\in [1,\,d-1]$, and we choose each $A_i$ in such a way that  $\mathsf v_h(A_i)$ is minimal. We continue with the following assertion.

\begin{enumerate}
\item[{\bf A.}\,] For each $i \in [1, d-1]$, we have $\mathsf v_h(A_i)< \mathsf v_h(U)\le \frac{\ord(h)}{2}$.
\end{enumerate}

\smallskip
{\it Proof of \,{\bf A}}.\
Assume to the contrary that there is an $i \in [1, d-1]$ such that $\mathsf v_h (A_i) \ge \mathsf v_h (U)$. Then
\[
h \notin F = \{h' \in \supp (U) \mid \mathsf v_{h'} (A_i) < \mathsf v_{h'} (U) \} \quad \text{and} \quad
U \, \big| \,A_i \prod_{h'\in F} {h'}^{\ord (h')} \,.
\]
Hence $A_i \prod_{h'\in F} {h'}^{\ord (h')} = U B_i$ for some zero-sum sequence $B_i$. By Lemma \ref{4.4} (items 2. and 3.),  $B_i$ is an atom with $i=\mathsf v_g(A_i)=\mathsf v_g(B_i)$ and with $\mathsf k (B_i)=\mathsf k (A_i)>1$. Since $\mathsf v_h(A_i)>\mathsf v_h(B_i)$, this is a  contradiction to the choice of $A_i$. \qed{(Proof of {\bf A})}

\medskip
Let $j \in [1, d-1]$ be such that  $\mathsf k(A_j)=\min \{ \mathsf k (A_1), \ldots, \mathsf k (A_{d-1}) \}$.

Suppose that $j\ge 2$. Let $V_1, \ldots,V_t$ be all the atoms with $V_s \t A_1A_{j-1}$ and $|\supp(V_s)|=1 $ for all $s\in [1,\,t]$. Then $B=A_1A_{j-1}(V_1  \cdot \ldots \cdot V_t)^{-1}$ is an atom by Lemma \ref{4.4}.2.
Since $\mathsf v_g(A_1A_{j-1})=j<\ord(g)$, $\mathsf v_h(A_1A_{j-1})<\ord(h)$, and $\mathsf v_f (A_1A_{j-1}) < 2 \ord (f)$ for all $f \in G_0 \setminus \{g,h\}$, it follows that $t\le |G_0|-2 \le  r-1$. Since $\min \Delta(G_0)=r-1$ and  $A_1A_{j-1}=V_1 \cdot \ldots \cdot V_t B$, we must have $t=1$. Therefore $\mathsf k(A_1)+\mathsf k(A_{j-1})=1+\mathsf k(B)$ whence $\mathsf k (B) < \mathsf k (A_{j-1})$. Since
\[
\mathsf v_g(B)= \mathsf v_g (V_1B)=\mathsf v_g (A_1 A_{j-1})=j=\mathsf v_g (A_j) \,,
\]
Lemma \ref{4.4}.3 implies that  $\mathsf k(B)=\mathsf k(A_j)=\min \{ \mathsf k (A_1), \ldots, \mathsf k (A_{d-1}) \} $, a contradiction.

Suppose that $j=1$. Let $V_1, \ldots, V_t$ be all the atoms with $V_s \t A_2A_{d-1}$ and $|\supp(V_s)|=1 $ for all $s\in [1,\,t]$. Then $B=A_2A_{d-1}(V_1  \cdot \ldots \cdot V_t)^{-1}$ is an atom by Lemma \ref{4.4}.2.
Since $\mathsf v_g(A_2A_{d-1})=d+1<\ord(g)$, $\mathsf v_h(A_2A_{d-1})<\ord(h)$, and $\mathsf v_f (A_1A_{j-1}) < 2 \ord (f)$ for all $f \in G_0 \setminus \{g,h\}$, it follows that $t \le |G_0|-2 \le  r-1$. Since $\min \Delta(G_0)=r-1$ and  $A_2A_{d-1}=V_1 \cdot \ldots \cdot V_t B$, we must have $t=1$. Therefore $\mathsf k(A_2)+\mathsf k(A_{d-1})=1+\mathsf k(B)$ whence $\mathsf k (B) < \mathsf k (A_{2})$. Since
\[
\mathsf v_g(B)= \mathsf v_g (V_1B)= \mathsf v_g (A_2A_{d-1}) = d+1 \equiv 1 = \mathsf v_g (A_1) \mod d \,,
\]
Lemma \ref{4.4}.3 implies that  $\mathsf k(B)=\mathsf k(A_1)=\min \{ \mathsf k (A_1), \ldots, \mathsf k (A_{d-1}) \} $, a contradiction.
\end{proof}

\medskip
In the following remark we provide the first example of a minimal non-half-factorial subset $G_0$ with $\min \Delta (G_0)=\max \Delta^* (G)$ which is not simple. Furthermore, we provide an example that the structural statement given in Theorem \ref{4.5}.4 does not hold without the assumption that the exponent is odd.

\medskip
\begin{remarks} \label{4.6}~
Following Schmid, we say that a nonempty subset $G_0 \subset G \setminus \{0\}$ is  {\it simple} if there exists some $g \in G_0$ such that $G_0 \setminus \{g\}$ is independent, $g \in \langle G_0 \setminus \{g\} \rangle$ but $g \notin \langle E \rangle$ for any subset $E \subsetneq G_0 \setminus \{g\}$.

If $G_0$ is a simple subset, then $|G_0| \le \mathsf r^* (G)+1$ and $G_0$ is indecomposable. Moreover, if $G_1 \subset G$ is a subset such that any proper subset of $G_1$ is independent, then there is a subset $G_0$ and a transfer homomorphism $\theta \colon \mathcal B (G_1) \to \mathcal B (G_0)$ where $G_0 \setminus \{0\}$ is simple or independent (for all this see \cite[Section 4]{Sc04a}). Furthermore, Theorem 4.7 in \cite{Sc04a} provides an intrinsic description of the sets of atoms  of a simple set.

In elementary $p$-groups, every minimal non-half-factorial subset is simple (\cite[Lemma 4.4]{Sc04a}), and so far there are no examples of minimal non-half-factorial sets $G_0$ with $\min \Delta (G_0) = \max \Delta^* (G)$ which are not simple.

\smallskip
1. Let $G=C_9^{r-1}\oplus C_{27}$ with $r\ge 26$, and let $(e_1,\ldots,e_r)$ be a basis of $G$ with $\ord (e_i)=9$ for $i\in [1,\,r-1]$ and $\ord(e_r)=27$. Then $\max \Delta^* (G)=r-1$ by Theorem \ref{1.1}.
We set $G_0=\{3e_1,\ldots,3e_{r-1},e_r,g\}$ with $g=e_1 +\ldots+e_r$. Then $(e_r, g)$ is not independent,  $G_0 \setminus \{g\}$ and $G_0 \setminus \{e_r\}$ are independent, but $g \notin \langle G_0 \setminus \{g\}\rangle$ and $e_r \notin \langle G_0 \setminus \{e_r\}\rangle$.  Therefore $G_0$ is not simple. It remains to show that $\min \Delta (G_0) \ge r-1$. Then $G_0$ is minimal non-half-factorial and $\min \Delta (G_0) = r-1$ because $\max \Delta^* (G) = r-1$.

We have
\begin{align*}
W_1=\{A\in \mathcal A(G_0)\mid \mathsf k(A)=1\}=\{&(3e_1)^3, \ldots, (3e_{r-1})^3,e_r^{27},g^{27},g^9e_r^{18}, g^{18}e_r^9\},\\  W_2=\{A\in \mathcal A(G_0)\mid \mathsf k(A)>1\}=\{&A_3=g^3e_r^{24}(3e_1)^2\cdot \ldots \cdot (3e_{r-1})^2, A_6=g^6e_r^{21}(3e_1)\cdot \ldots \cdot (3e_{r-1}),\\
  & A_{12}=g^{12}e_r^{15}(3e_1)^2\cdot \ldots \cdot (3e_{r-1})^2, A_{15}=g^{15}e_r^{12}(3e_1)\cdot \ldots \cdot (3e_{r-1}), \\   &A_{21}=g^{21}e_r^6(3e_1)^2\cdot \ldots \cdot (3e_{r-1})^2, A_{24}=g^{24}e_r^3(3e_1)\cdot \ldots \cdot (3e_{r-1})\}
\end{align*}
and $\mathsf k(A_3)=\mathsf k(A_{12})=\mathsf k(A_{21})=(2r+1)/3$, $\mathsf k(A_6)=\mathsf k(A_{15})=\mathsf k(A_{24})=(r+2)/3$.
For any $d\in \Delta (G_0)$, there exists a $B \in \mathcal B (G_0)$ such that $B$ has two such factorizations, say
 \[
    B=U_1 \cdot \ldots \cdot U_s V_1 \cdot \ldots \cdot V_t  W_1 \cdot \ldots \cdot W_u=X_1 \cdot \ldots \cdot X_{s'} Y_1 \cdot \ldots \cdot Y_{t'}  Z_1 \cdot \ldots \cdot Z_{u'}\,
 \]
where all $U_i, V_j, W_k, X_{i'}, Y_{j'}, Z_{k'}$ are atoms, $s, t, u, s',t',u' \in \N_0$ with $d=(s+t+u)-(s'+t'+u')$, $\mathsf k (U_1)=\ldots=\mathsf k (U_s)=\mathsf k (X_1)= \ldots = \mathsf k (X_{s'})=\frac{2r+1}{3}$,  $\mathsf k (V_1)= \ldots = \mathsf k (V_t)=\mathsf k (Y_1) = \ldots = \mathsf k (Y_{t'})= (r+2)/2$, and $\mathsf k (W_1)=\ldots=\mathsf k (W_u)=\mathsf k (Z_1)= \ldots = \mathsf k (Z_{u'})=1$. This implies that
 \[
 \mathsf k(B)=s(\frac{2r+1}{3})+t(\frac{r+2}{3})+u=s'(\frac{2r+1}{3})+t'(\frac{r+2}{3})+u'
 \]
and $\mathsf v_{3e_1} (B)\equiv 2s+t\equiv 2s'+t' \mod 3$.
Since $d=(s+t+u)-(s'+t'+u')=\frac{r-1}{3}((t'-t)+2(s'-s))>0$, we obtain that  $(t'-t)+2(s'-s)\ge 3$ and hence $d\ge r-1$.

\medskip
2. We provide an example of a minimal non-half-factorial LCN-set with $\min \Delta (G_0) = \max \Delta^* (G)$ in a group $G$ of even exponent which has no element $g \in G_0$ such that $ G_0 \setminus \{g\}$ is independent. In particular, $G_0$ is not simple and the assumption in Theorem \ref{4.5}.4, that the exponent of the group is odd, cannot be cancelled.

Let $G=C_2^{r-2}\oplus C_4\oplus C_4$ with $r\ge3$, and let $(e_1, \ldots, e_r)$ be a basis of $G$ with $\ord (e_i)=2$ for $i\in [1,\,r-2]$ and $\ord(e_{r-1})=\ord(e_r)=4$.
We set $G_0=\{e_1, \ldots,e_{r-3},e_{r-2}+e_{r-1},e_{r-1},e_r,g\}$ with $g=e_1+ \ldots+e_{r-2}+e_r$. Since $(e_{r-2}+e_{r-1},e_{r-1})$ is dependent and $(e_r,g)$ is dependent, we obtain that there is no $h\in G_0$ such that $G_0\setminus\{h\}$ is independent. We have
\begin{align*}
W_1=\{A\in \mathcal A(G_0)\mid \mathsf k(A)=1\}=\{&(e_1)^2, \ldots, (e_{r-3})^2,(e_{r-2}+e_{r-1})^{4},(e_{r-1})^4,e_r^4,g^4,\\
&(e_{r-2}+e_{r-1})^{2}(e_{r-1})^2,g^2e_r^2\},\\
 W_2=\{A\in \mathcal A(G_0)\mid \mathsf k(A)>1\}=\{&A_1=ge_r^3(e_{r-2}+e_{r-1})e_{r-1}^3e_1\cdot \ldots \cdot e_{r-3},\\
 & B_1=ge_r^3(e_{r-2}+e_{r-1})^3e_{r-1}e_1\cdot \ldots \cdot e_{r-3},\\
  &A_3=g^3e_r(e_{r-2}+e_{r-1})e_{r-1}^3e_1\cdot \ldots \cdot e_{r-3},\\
  & B_3=g^3e_r(e_{r-2}+e_{r-1})^3e_{r-1}e_1\cdot \ldots \cdot e_{r-3}\}
\end{align*}
and $\mathsf k(A_1)=\mathsf k(A_3)=\mathsf k(B_1)=\mathsf k(B_3)=(r+1)/2$.  Theorem \ref{1.1} implies that $\max \Delta^* (G) = r-1$, and thus it remains to show that  $\min \Delta (G_0)  = r-1$.

For any $d\in \Delta (G_0)$, there exists a $B \in \mathcal B (G_0)$ such that $B$ has two such factorizations, say
\[
B = U_1 \cdot \ldots \cdot U_sV_1 \cdot \ldots \cdot V_t = X_1 \cdot \ldots \cdot X_uY_1 \cdot \ldots \cdot Y_v \,
\]
where all $U_i, V_j, X_k, Y_l$ are atoms, $s, t, u, v \in \N_0$ with $d=u+v-(s+t)$, $\mathsf k (U_1)=\ldots=\mathsf k (U_s)=\mathsf k (X_1)= \ldots = \mathsf k (X_u)=1$, and $\mathsf k (V_1)= \ldots = \mathsf k (V_t)=\mathsf k (Y_1) = \ldots = \mathsf k (Y_v)= (r+1)/2$. This implies that
\[
\mathsf k (B) = s + t \frac{r+1}{2} = u + v \frac{r+1}{2}
\]
and $\mathsf v_{g}(B)\equiv t \equiv v \mod 2$. Since $d=(v+u) - (s+t) = (t-v) \frac{r-1}{2}>0$, we obtain that $t-v \ge 2$ and hence $d\ge r-1$.
\end{remarks}

%%%%%%%%%%%%%%%%%%%%%%%%%%%%%%%%%%%%%%%%%%%%%%%%%%%%%%%%%%%%%%%%%%%%%%%%%
%% BIBLIOGRAPHY  %%%%%%%%%%%%%%%%%%%%%%%%%%%%%%%%%%%%%%%%%%%%%%%%%%%%%%%%
%%%%%%%%%%%%%%%%%%%%%%%%%%%%%%%%%%%%%%%%%%%%%%%%%%%%%%%%%%%%%%%%%%%%%%%%%

\providecommand{\bysame}{\leavevmode\hbox to3em{\hrulefill}\thinspace}
\providecommand{\MR}{\relax\ifhmode\unskip\space\fi MR }
\providecommand{\MRhref}[2]{%
  \href{http://www.ams.org/mathscinet-getitem?mr=#1}{#2}
}
\providecommand{\href}[2]{#2}

\end{document}